\documentclass{amsart}

  \newcommand\ma[1]{{\color{black} #1}}
 \newcommand\ibt[1]{{\color{black} #1}}
 
\usepackage{floatrow}
\usepackage{caption}
 \usepackage[foot]{amsaddr}
\let\oldtocsection=\tocsection

\let\oldtocsubsubsection=\tocsubsubsection

\renewcommand{\tocsection}[2]{\vspace{0.5em}\hspace{0em}\oldtocsection{#1}{#2}}
\renewcommand{\tocsubsubsection}[2]{\vspace{0.5em}\hspace{2em}\oldtocsubsubsection{#1}{#2}}
\usepackage{graphicx,cancel,xcolor,hyperref,comment,graphicx,geometry}

\usepackage{amsopn}

\DeclareMathOperator{\divv}{div}

\usepackage{tikz}
\usepackage{graphicx,url,etoolbox}
\usepackage{lipsum}
\usepackage{dsfont}
 \usepackage{amssymb}

\usepackage{graphicx,cancel,xcolor,hyperref,comment,graphicx,geometry}

\usepackage{tikz}
\usetikzlibrary{decorations.pathreplacing}
\usepackage{graphicx,url,etoolbox}
\usepackage{lipsum}
\usepackage{dsfont}

\usepackage{fancyhdr}
\usepackage{lastpage}

\newtheorem{theoreme}{Theorem}

\newtheorem{rem}[theoreme]{Remark}
\theoremstyle{definition}

\numberwithin{equation}{section}

 \renewenvironment{proof}{{\bfseries \noindent Proof.}}{\demo}
\newcommand\xqed[1]{%
  \leavevmode\unskip\penalty9999 \hbox{}\nobreak\hfill
  \quad\hbox{#1}}
\newcommand\demo{\xqed{$\square$}}

\hypersetup{bookmarks, colorlinks, urlcolor=blue, citecolor=blue, linkcolor=blue, hyperfigures, pagebackref,
   pdfcreator=LaTeX, breaklinks=true, pdfpagelayout=SinglePage, bookmarksopen=true,bookmarksopenlevel=2}

\usepackage{comment}

\def\R{\mathbb R}

\def\la {{\lambda}}

\newcommand {\nc}   {\newcommand}
\nc {\be}   {\begin{equation}} \nc {\ee}   {\end{equation}} \nc
{\beq}  {\begin{eqnarray}} \nc {\eeq}  {\end{eqnarray}} \nc {\beqs}
{\begin{eqnarray*}} \nc {\eeqs} {\end{eqnarray*}}
\def\edc{\end{document}}

\providecommand{\abs}[1]{\lvert#1\rvert}
   
\usepackage{tikz}
\usetikzlibrary{decorations.pathmorphing,patterns,scopes,intersections,calc}
\usepackage{caption}
\usepackage{float}
\usepackage{soul}

\newcounter{dummy} 
\numberwithin{dummy}{section}
\newtheorem{Theorem}[dummy]{Theorem}

\newtheorem{example}[dummy]{Example}

\newtheorem{Lemma}[dummy]{Lemma}

\newtheorem{Proposition}[dummy]{Proposition}
\newtheorem{Remark}[dummy]{Remark}

\newtheorem{Hypothesis}[dummy]{Hypothesis}
\numberwithin{equation}{section}

\DeclareUnicodeCharacter{2212}{-}

\begin{document}
\title[\fontsize{7}{9}\selectfont  ]{Energy decay rate of a transmission system governed by degenerate wave equation with drift and under heat conduction with memory effect}
\author{Mohammad Akil$^1$ , Genni Fragnelli$^2$, Ibtissam Issa$^3$}

\address{$^1$ Universit\'e Polytechnique Hauts-de-France, C\'ERAMATHS/DEMAV, Le Mont Houy 59313 Valenciennes Cedex 9-France.}

\address{$^2$ Department of Ecology and Biology, Tuscia University, Largo dell’Universit\'a, 01100 Viterbo - Italy}
\address{$^3$ Universit\'a  degli studi di Bari Aldo Moro-Italy, Dipartimento di Matematica, Via E. Orabona 4, 70125 Bari - Italy}

\email{mohammad.akil@uphf.fr}
\email{genni.fragnelli@unitus.it}
\email{ibtissam.issa@uniba.it}


\begin{abstract}
In this paper, we investigate the stabilization of transmission problem of degenerate wave equation and heat equation under Coleman-Gurtin heat conduction law or Gurtin-Pipkin law with memory effect. We investigate the polynomial stability of this system when employing the Coleman-Gurtin heat conduction,  establishing a decay rate of type $t^{-4}$. Next, we demonstrate exponential stability in the case when Gurtin-Pipkin heat conduction is applied.
\end{abstract}

\maketitle
\pagenumbering{roman}
\maketitle
\tableofcontents

\pagenumbering{arabic}
\setcounter{page}{1}

\section{\textbf{Introduction}}

\vspace{0.3cm}

In the fields of science and engineering, it is common to encounter numerous models that involve the coupling of heat equations and wave equations.  Apart from the underlying mathematical attraction, the primary driving force for researching these systems is their potential applications. In the study conducted in \cite{EdwardNaghdi1991} in 1991, a thermoelasticity model of type II is proposed. This model offers a framework for understanding and analyzing thermoelastic phenomena.  The models under consideration have been extensively examined in previous years. For a comprehensive analysis of the asymptotic behavior of these systems,  see \cite{ZhongZuazua2017} and the relevant literature cited therein.  However, as mentioned in \cite{Biot1941},  in 1941, the phenomenon of an earthquake is mathematically represented by a three-dimensional nonlinear coupled heat-wave partial differential equation (PDE). In the study conducted in \cite{DiegoYury2022}, a simplified one-dimensional representation of an earthquake model is examined. The authors propose the utilization of a sliding mode controller to mitigate instabilities, specifically to prevent the occurrence of earthquakes.  Fluid-structure interaction can also be represented mathematically as a coupled system of such a type.  Such system can be viewed as linearisations of more complex fluid-structure models arising in fluid mechanics; for an in-depth understanding see for instance  \cite{AvalosTriggiani2007, ZhangZuazua2006}.   In thermoelasticity, the stability of a system is important for preventing thermal buckling  where a solid structure collapses due to thermal expansion and to prevent deformations or failures in structures subjected to both thermal and mechanical loads. In \cite{Filippo2023DCDS} the authors worked  to produce a family of equations describing the evolution of the temperature in a rigid heat conductor by using means of successive approximations of the Fourier law, via memory relaxations and integral perturbations.  In \cite{Pata2010}, the authors introduce a novel method for mathematically analyzing equations with memory that is based on the idea of a state that is, the initial configuration of the system that can be clearly identified by knowing the dynamics of the future. They use the abstract form of an equation resulting from linear viscoelasticity as a model.

\vspace{0.3cm}

In \cite{BattyDavid2016}, in 2016, the authors considered the wave-heat system with finite wave and heat parts coupled through the boundary with Neumann boundary for the wave part. They show that the optimal energy decay rate of type $t^{-4}$.   A similar result is obtained in \cite{ZhangZuazua2004} for a similar one-dimensional problem in which the wave part satisfies a Dirichlet boundary condition rather than a Neumann boundary condition at the endpoint of the wave. A work that considered extending the heat part of the coupled wave-heat system is studied in 2020, in \cite{David}. The authors establish a sharper rate ($t^{-2}$) which shows that extending the heat part to infinity slows the rate of energy decay by a factor of $t^2$. The crucial difference between the infinite and the finite systems is that the damping provided by the heat part is significantly weaker in the infinite case.  
\\

The classical linear heat equation is conventionally derived from Fourier's law, the constitutive equation, and the first law of thermodynamics. However, this conventional theory exhibits two primary shortcomings. Firstly, it does not incorporate memory effects, which are evident in certain materials. Secondly, it postulates that all thermal perturbations at a specific location within the material are instantaneously transmitted throughout the entire body, implying that all disturbances propagate at an infinite velocity. The study of heat conduction under diverse non-Fourier heat flux laws has been evolving since the 1940s. Let's denote the heat flux vector as $\mathbf{q}$.  As per the Gurtin-Pipkin theory \cite{Gurtin1968}, the linearized constitutive equation for $\mathbf{q}$ is established as follows:
\begin{equation}\label{GP}
\mathbf{q}(t)=-\int_0^{\infty}g(s)\theta_x(t-s)ds,
\end{equation}
where $g$ is the heat conductivity relaxation kernel. The  presence of convolution term in \eqref{GP} entails finite propagation speed of heat conduction, and consequently, the equation is of hyperbolic type.  Note that \eqref{GP} reduces to the classical Fourier law when $g$ is the Dirac mass at zero. Furthermore, if we take $g$ as a prototype kernel
$g(t) = e^{-kt}, k > 0,$ and differentiate \eqref{GP} with respect to $t$, we can (formally) arrive at the so-called Cattaneo–Fourier law  \cite{Cattaneo}
\begin{equation}\label{Cattaneo–Fourier}
\mathbf{q}_t(t)+k\mathbf{q}=-\theta_x(t).
\end{equation}
On the other hand, when the heat conduction is due to the Coleman–Gurtin theory \cite{Coleman1967}, the heat flux $\mathbf{q}$ depends on both the past history and the instantaneous of the gradient of temperature:
\begin{equation}\label{CG}
\mathbf{q}(t)=-\beta\theta_x(x,t)-\int_0^{\infty}g(s)\theta_x(x,t-s)ds,
\end{equation}
where $\beta>0$ is the instantaneous diffusivity coefficient. \\
Both the Coleman-Gurtin and Pipkin-Gurtin heat conduction laws incorporate memory effects into heat transfer equations. The Coleman-Gurtin model introduces memory effects through a time derivative term, providing a nuanced description of how materials retain thermal memory, particularly valuable in scenarios with rapid temperature changes and non-classical heat conduction. While, the Pipkin-Gurtin model employs fractional calculus to describe memory effects, making it well-suited for materials with complex structures or fractal-like properties, allowing it to capture non-local heat conduction and extended long-range interactions. For example,  thermoacoustic devices like engines and refrigerators involve the interaction of acoustic waves and heat transfer. 
\\

In 2014, in \cite{Zhang-ZAMP} Zhang studies the stability of an interaction system comprised of a wave equation and a heat equation with memory.  An exponential stability of the interaction system is obtained when the hereditary heat conduction is of Gurtin–Pipkin type and she showed the lack of uniform decay of the interaction system when the heat conduction law is of Coleman–Gurtin type (see also \cite{Pata2001} in 2001 where the model considered uses a Gurtin-Pipkin linearized heat flux rule to match the energy balance with a nonlinear time-dependent heat source).  In \cite{DellOro2020}, the author studies the stability of Bresse and Timoshenko systems with hyperbolic heat conduction. As a  first step, the Bresse–Gurtin–Pipkin system is studied, providing a necessary and sufficient condition for the exponential stability and the optimal polynomial decay rate when the condition is violated. As a second step, in  \cite{DellOro2020}, the Timoshenko–Gurtin–Pipkin system is considered and the optimal polynomial decay rate is found.
Later, in \cite{Filipo2023}, the authors study the asymptotic behavior of solutions of a one-dimensional coupled wave-heat system with Coleman-Gurtin thermal law. They prove an optimal polynomial decay rate of type $t^{-2}$.  In \cite{Akil2022} the author studies the stability of a system of piezoelectric beams under (Coleman or Pipkin)–Gurtin
thermal laws with magnetic effect. In particular, as a first step, the author explores the piezoelectric Coleman-Gurtin system, achieving exponential stability; as a second step, he considers the piezoelectric Gurtin–Pipkin system, establishing a polynomial energy decay rate of type $t^{-2}$ .\\
As far as our understanding goes, there appears to be no prior research on the examination of a system that combines a degenerate wave equation and a heat equation with memory. In pursuit of this objective, this paper explores this system, examining various scenarios involving different heat conduction types. 
Several problems arising in Physics and Biology (see \cite{Karachalios2005}), Biology (see \cite{Boutaayamou2020,  Fragnelli+2020}), as well as Mathematical Finance (see \cite{Hagan}), are governed by degenerate parabolic equations.
The existing literature focused on controlling and stabilizing the nondegenerate wave equation using diverse damping methods is notably extensive. This fact can be observed in the substantial number of works cited, as exemplified by \cite{CHENTOUF, Conrad1993, dAndraNovel1994} and the references mentioned within.\\
Lately, the subject of controllability and stability in degenerate hyperbolic equations has gained significant attention, with various advancements made in recent years.  In this field the most important paper is \cite{Cannarsa} (see also the arxiv version of 2015), where a general degenerate function is considered (see  also \cite{Gueye},  \cite{Zhang2015}, and the references mentioned within).
 In \cite{Alabau2006}, the authors establish a Carleman estimate for the one-dimensional degenerate heat equation and investigate the null controllability of the semilinear degenerate parabolic equation over the interval $[0, 1]$. Meanwhile, in the work described in \cite{Fragnelli2016}, the authors formulate Carleman estimates for singular/degenerate parabolic Dirichlet problems, taking into account degeneracy and singularity within the spatial domain's interior (see also \cite{Fragenllibook}).\\
Recently, in \cite{fragnelli2022linear},  the authors consider  a degenerate wave equation in one dimension, with drift and in presence of a leading degenerate operator which is in non-divergence form with homogeneous Dirichlet boundary condition where the degeneracy occurs and a boundary damping is considered at the other endpoint.  In particular, they prove uniform exponential decay under some conditions for the  solutions. Later, a boundary controllability problem for a similar system is considered in \cite{Boutaayamou2023}. In particular, the authors study the controllability of the system  by providing some conditions for the boundary controllability of the solution of the associated Cauchy problem at a sufficiently large time. Recently in 2023, in \cite{akil2023stability} the authors consider two problems; the first one a one-dimensional degenerate wave equation with degenerate damping, incorporating a drift term and a leading operator in non-divergence form. In the second problem they consider a system that couples degenerate and non-degenerate wave equations, connected through transmission and subject to a single dissipation law at the boundary of the non-degenerate equation. In both scenarios, they reached exponential stability results. Recently, in \cite{ZhongjieJDE},  a system consisting of one wave equation and one degenerate heat equation in two connected regions is considered where the coupling is through certain transmission conditions. The authors consider two cases, the first is on rectangular domain where the system is given as follows
\begin{equation}
\left\{\begin{array}{ll}
u_{tt}(x,y,t)-\Delta u(x,y,t)=0,& (x,y) \in \Omega_1, t>0,\\
w_t(x,y,t)-\divv{(a\nabla w)}(x,y,t)=0,& (x,y) \in \Omega_2, t>0,\\
u(-1,y,t)=w(1,y,t)=0&y\in (0,1),  t>0,\\
u(x,0,t)=u(x,1,t)=0&x\in (-1,0), t>0,\\
w(x,0,t)=u(x,1,t)=0&x\in (0,1), t>0,\\
u(x,y,0)=u_0(x,y),u_t(x,y,0)=v_0(x,y), &(x,y) \in \Omega_1, t>0,\\
w(x,y,0)=w_0(x,y), u_t(x,0)=v_0(x),&(x,y) \in \Omega_2, t>0,
\end{array}\right.
\end{equation}
with the transmission conditions at the interface line $\Gamma$ : $x = 0$, given by
\[u_t(0,y,t)=w(0,y,t), u_x(0,y,t)=(aw_x)(0,y,t),\quad y\in (0,1) \text{ and }  t>0.\]
Here the variable heat-diffusion coefficient $a(x, y)$ degenerates near the interface line $x = 0$, as follows: $a(x,y)=mx^{\alpha}, x\in(0,1]$ with constants $0\leq \alpha<1, m>0$ and where $\Omega_1=(-1,0)\times (0,1)$ and $\Omega_2=(0,1)\times (0,1)$. 
In this case,  the authors establish an explicit polynomial decay rate of type $t^{\frac{-1}{3-\alpha}}$ for the solutions to the system. This decay rate depends
only on the degeneration degree of the diffusion coefficient of the heat equation near the
interface line. The second case is the one-dimensional case where they reached polynomial decay rate with rate $t^{-\frac{2-\alpha}{1-\alpha}}$.\\
This paper examines a one-dimensional coupled system wherein a wave equation and a heat equation with memory are interconnected via transmission conditions established at the interface.  Our investigation will focus on the system's stability in various cases of heat conduction, which are dependent on the parameter $m$.  The system is given as following
\begin{equation}\label{Eq1}\tag{DW-H$_{m}$}
\left\{\begin{array}{lll}
u_{tt}-a(x)u_{xx}-b(x)u_{x}=0,& (x,t) \in (0,1)\times\R^+_{\ast},\\[0.1in]\displaystyle
y_t-c(1-m)y_{xx}-cm\int_{0}^{\infty}g(s)y_{xx}(x,t-s)ds=0,& (x,t) \in (1,2)\times\R^+_{\ast},\\[0.1in]
u_t(1,t)=y(1,t), & t\in \R^+_{\ast},\\[0.1in]\displaystyle
\eta(1) u_x(1,t)=c(1-m)y_x(1,t)+cm\int_{0}^{\infty}g(s)y_{x}(1,t-s)ds, & t\in \R^+_{\ast},\\[0.1in]
u(0,t)=y(2,t)=0, & t\in \R^+_{\ast},\\[0.1in]
u(x,0)=u_0(x), u_t(x,0)=v_0(x),& x\in (0,1),\\[0.1in]
y(x,0)=y_0(x), y(x,-s)=\varphi_0(x,s),  & x\in (1,2), s>0.
\end{array}\right.
\end{equation}
Here $a,b\in C^0([0,1])$, with $a>0$ on $(0,1]$, $a(0)=0$ and $\frac{b}{a}\in L^1(0,1)$,  $u_0$, $v_0$, $y_0$ are assigned data, $c$ is a strictly positive constant and $\varphi_0$ accounts for the so-called initial past history of $y$. The convolution kernel $g :[0,\infty[\to [0,\infty)$ is a convex integrable function (thus non-increasing and vanishing at infinity) of unit total mass, taking the explicit form
$$ g(s)=\int_s^{\infty}\mu(r)dr, s\geq 0,  $$
where the memory kernel $\mu : (0,\infty)\to [0,\infty)$ satisfies the following conditions
\begin{equation*}\label{H}\tag{H}
\left\{
\begin{array}{ll}
\displaystyle
\mu\in L^1(0,\infty)\cap C^1(0,\infty)\quad\text{with}\quad \int_0^{\infty}\mu(r)dr=g(0)>0, \, \mu(0)=\lim_{s\to 0}\mu(s)<\infty,\\
\displaystyle
\mu \quad \text{satisfies the Dafermos condition}\quad \mu^{\prime}(s)\leq -K_{\mu}\,\mu(s); \quad K_{\mu}>0.
\end{array}\right.
\end{equation*}
Also, let us recall the well-known absolutely continuous weight function
$$\eta(x):=\exp\left\lbrace\int _{\frac{1}{2}}^{x}\dfrac{b(s)}{a(s)}ds\right\rbrace, \quad x\in [0,1]$$
introduced by Feller in a related context \cite{Feller}.  We set the function $\sigma(x):=\dfrac{a(x)}{\eta(x)},$ which is a continuous function in $[0, 1]$, independent of the possible degeneracy of $a$. Moreover, observe that if $u$ is a sufficiently smooth function, e.g. $u\in W^{2,1}_{loc}(0,1)$, then we can write $Bu := au_{xx} + bu_x$
as
$$Bu=\sigma(\eta u_x)_x.$$
The degeneracy at $x = 0$ is measured by the parameter $K$ defined by
\begin{equation}\label{Degeneracy}
K_a:=\sup _{x\in (0,1]}\dfrac{x\abs{a^{\prime}(x)}}{a(x)}.
\end{equation}
We say that $a$ is weakly degenerate at $0$, (WD) for short, if 
\begin{equation}\label{WD}\tag{WD}
a\in C^0[0,1]\cap C^1(0,1] \quad  \text{and} \quad K_a\in (0,1)
\end{equation}
(for example $a(x)=x^K$, with $K \in (0,1)$)
and we say that $a$ is strongly degenerate at $0$, (SD) for short, if 
\begin{equation}\label{SD}\tag{SD}
a\in C^1[0,1] \quad \text{and}\quad K_a\in [1,2)
\end{equation}
(for example $a(x)=x^K$, with $K \in [1,2)$).
Here we assume $K_a<2$ since it is essential in the calculation that will be conducted below. \\

\noindent In the model \eqref{Eq1}, $m \in (0, 1)$ is a fixed parameter and the temperatures obey the parabolic hyperbolic law introduced by Coleman and Gurtin in \cite{Coleman1967}. For the boundary cases for $m$:\\
$\bullet$ For $m=0$: The system corresponds to Fourier's law.\\
$\bullet$ For  $m=1$: The system corresponds to Gurtin-Pipkin heat conduction law.\\
The main novelty in this paper is the consideration of the transmission problem of degenerate wave equation and under heat conduction (Coleman or Pipkin)-Gurtin or Fourier law.  Since we investigate the degenerate wave equation in this work, so we can approach the system considered in \cite{Zhang-ZAMP} if we choose specific functions $a(x)=1$ and $b(x)=0$.  Additionally, this study extends the findings presented in \cite{Filipo2023} by examining the degenerate wave equation as opposed to the classical wave equation, while also incorporating the three thermal laws mentioned. The optimality of the decay rate in the context of the Coleman-Gurtin or Fourier law is established when choosing $a(x)=1$ and $b(x)=0$ (as demonstrated in \cite{Filipo2023}). However, for the system described in \eqref{Eq1}, we propose that the energy decay rate is optimal. \\
This paper is organized as follows: in the first section we give some preliminary results and we establish the well posedness of the system  \eqref{Eq1} if $m \in [0,1]$ using a semigroup approach,  after re-framing the system into an evolution system.  In Section \ref{Strong Stability}, we investigate the stability of the system under (Coleman or Pipkin)-Gurtin thermal law in the case $m \in [0,1]$. Next, in Section \ref{SectionPoly}, we establish a polynomial stability under Coleman-Gurtin heat conduction law with decay rate of type $t^{-4}$ if $m \in [0,1)$. Section \ref{SectionExp}, we prove the exponential stability of the system under Gurtin-Pipkin heat conduction law with the parameter $m=1$. The paper ends with a section devoted to recall some results needed for the proofs.

\section{\textbf{Preliminaries, Functional spaces and Well-Posedness}}\label{Section-Well Posedness}
\begin{Hypothesis}\label{hyp1}
Functions $a$ and $b$ are continuous in $[0,1]$ and such that $\dfrac{b}{a}\in L^1(0,1)$.
\end{Hypothesis}

\begin{Hypothesis}\label{hyp3}
Hypothesis \ref{hyp1} holds.  In addition, $a$ is such that $a(0)=0, a>0$ on $(0,1]$ and there exists $\alpha>0$ such that the function 
\begin{equation}\label{rq}
x\rightarrow\dfrac{x^{\alpha}}{a(x)}
\end{equation}
is non-decreasing in a right neighborhood of $x=0$.
\end{Hypothesis}

\begin{rem}
\begin{enumerate}\normalfont
\item If $a$ is (WD) or (SD), then \eqref{rq} holds for all $\alpha\geq K_a$ and for all $x\in (0,1)$.
\item We notice that, at this stage, a may not degenerate at $x = 0$.  However, if it is (WD) then $\dfrac{1}{a}\in L^1(0,1)$ and the assumption $\dfrac{b}{a}\in L^1(0,1)$ is always satisfied.  If $a$ is (SD) then $\dfrac{1}{a}\notin  L^1(0,1)$,  hence, if we want $\dfrac{b}{a}\in L^1(0, 1)$  then $b$ has to degenerate at $0$. In this case $b$ can be (WD) or (SD).
\end{enumerate}
\end{rem}
We start by introducing the following spaces.
$$L_{\frac{1}{\sigma}}^2(0,1):=\left\lbrace y\in L^2(0,1); \|y\|_{\frac{1}{\sigma}}<\infty\right\rbrace,  \quad\langle y,z\rangle_{\frac{1}{\sigma}}:=\int _0^1 \dfrac{1}{\sigma}y\bar{z}dx,\quad \text{for every} \quad y,z\in L_{\frac{1}{\sigma}}^2(0,1),$$
$$\ma{H_{\frac{1}{\sigma}}^1(0,1)}:=L_{\frac{1}{\sigma}}^2(0,1)\cap \ma{H^1(0,1)}, \quad \langle y,z\rangle_{1}:=\langle y,z\rangle_{\frac{1}{\sigma}}+\int _0^1\eta y_x\bar{z}_xdx,\quad \text{for every}  \quad y,z\in H_{\frac{1}{\sigma}}^1(0,1), $$
$$H^1_L(0,1)=\left\{u\in H^1(0,1); u(0)=0 \right\}, \quad \left<y,z\right>_{H^1_L(0,1)}=\left<y_x,z_x\right>_{L^2(0,1)}\quad \text{for every}  \quad y,z\in H_{L}^1(0,1),$$
and 
$$H_{\frac{1}{\sigma}}^2(0,1):=\left\lbrace y\in  H_{\frac{1}{\sigma}}^1(0,1); By\in L_{\frac{1}{\sigma}}^2(0,1) \right\rbrace\quad \langle y,z\rangle_{2}:=\langle y,z\rangle_{1}+\langle By,Bz\rangle_{\frac{1}{\sigma}}.$$
The previous inner products induce related respective norms
$$\|u\|^2_{\frac{1}{\sigma}}=\int_0^1 \frac{1}{\sigma}|u|^2dx, \quad \|u\|_{1}^2=\|u\|^2_\frac{1}{\sigma}+\int_0^1 \eta |u_x|^2dx\quad \text{and}\quad \|u\|_{2}^2=\|u\|^2_{1}+\int_0^1 \sigma |(\eta u_x)_x|^2dx.$$
Also, we consider the following spaces
$$H_{\frac{1}{\sigma},L}^1(0,1)= L_{\frac{1}{\sigma}}^2(0,1)\cap H^1_L(0,1)\quad\text{and}\quad H_{\frac{1}{\sigma},L}^2(0,1):=\left\{y\in H_{\frac{1}{\sigma},L}^1(0,1); By \in L_{\frac{1}{\sigma}}^2(0,1) \right\}$$
endowed with the previous inner products and related norms and we denote by $\|\cdot\|=\|\cdot\|_{L^2(0,1)}$. We also introduce
$$H^1_R(1,2):=\{y \in H^1(1,2): y(2)=0\}$$
endowed with the following norm
$$\|y\|^2_{H_R^1(1,2)}=\int_1^2|y_x|^2dx.$$
We also introduce the memory space $W$, defined by
$$W=L^2(\R^{+}, H^1_R(1,2))\quad \left<\gamma_1,\gamma_2\right>_W=cm\int_1^2\int_0^{\infty}\mu(s)\gamma_{1x}\overline{\gamma_{2x}}dsdx\quad \text{for all}  \quad \gamma_1,\gamma_2\in W.$$
Then,  we reformulate system \eqref{Eq1} using the definition of $\sigma$ and using the history framework of Dafermos. To this end, for $s > 0$, we consider the auxiliary function
 $$\gamma(x,s)=\int_0^{s}y(x,t-r)dr,\quad x\in (1,2)\,\text{and}\, s>0.
 $$
Now, system \eqref{Eq1} can be rewritten as 
\begin{equation}\label{system1}
\left\{\begin{array}{lll}
u_{tt}-\sigma(\eta u_{x})_x=0,& (x,t) \in (0,1)\times\R^+_{\ast},\\[0.1in]\displaystyle
y_t-c(1-m)y_{xx}-cm\int_{0}^{\infty}\mu(s)\gamma_{xx}(s)ds=0,& (x,t) \in (1,2)\times\R^+_{\ast},\\[0.1in]
\gamma_t+\gamma_s-y=0,& (x,t) \in (1,2)\times\R^+_{\ast},\\[0.1in]
u_t(1,t)=y(1,t), &t\in \R^+_{\ast},\\[0.1in]
\displaystyle
\eta(1) u_x(1,t)=c(1-m)y_x(1,t)+cm\int_{0}^{\infty}\mu(s)\gamma_{x}(1,s)ds, &t\in \R^+_{\ast},\\[0.1in]
u(0,t)=0, y(2,t)=0, &t\in \R^+_{\ast},\\[0.1in]
u(x,0)=u_0(x), u_t(x,0)=v_0(x),&x\in (0,1),\\[0.1in]
y(x,0)=y_0(x), y(x,-s)=\varphi_0(x,s),  & x\in (1,2), s>0.
\end{array}\right.
\end{equation}
Multiplying the first equation in \eqref{system1} by $\displaystyle\frac{1}{\sigma}\overline{u_t}$ and integrating over $(0,1)$, we get
\begin{equation}\label{e1}
\frac{1}{2}\frac{d}{dt} \left(\int_0^1 \frac{1}{\sigma}|u_t|^2dx+\int_0^1\eta |u_x|^2dx\right)-\Re\left(\eta(1)u_x(1)\overline{u_t}(1)\right)=0.
\end{equation}
Multiplying the second equation in \eqref{system1} by $\overline{y}$ and integrating over $(1,2)$, we obtain\begin{equation}\label{e2}
\begin{array}{l}
\displaystyle\frac{1}{2}\frac{d}{dt}\int_1^2|y|^2dx+c(1-m)\int_1^2|y_x|^2dx+\Re\left(c(1-m)y_x(1)\overline{y}(1)\right)+\Re\left(cm\int_1^2\int_0^\infty\mu(s)\gamma_x(s)\overline{y}_xdx\right)\\ 
\displaystyle
+\Re\left(cm\int_0^\infty\mu(s)\gamma_x(1,s)\overline{y}(1)ds\right)=0.
\end{array}
\end{equation}
Differentiating the third equation in \eqref{system1} with respect to $x$, we obtain
$$ \gamma_{xt}+\gamma_{xs}-y_x=0.$$
Multiplying the above equation by $cm\mu(s)\overline{\gamma}_x$, integrating over $(1,2)\times(0,\infty)$, we get
\begin{equation}\label{e3}
\begin{array}{ll}
\displaystyle
\frac{1}{2}\frac{d}{dt}cm\int_1^2\int_0^{\infty}\mu(s)|\gamma_x|^2dsdx-\frac{cm}{2}\int_1^2\int_0^{\infty}\mu^{\prime}(s)|\gamma_x|^2dsdx+\frac{cm}{2}\left[\int_1^2\mu(s)|\gamma_x|^2dx\right]_0^{\infty}\\
\displaystyle
=\Re\left(cm\int_1^2\int_0^\infty\mu(s)\overline{\gamma}_x(s){y}_xdx\right).
\end{array}
\end{equation}
Inserting \eqref{e3} in \eqref{e2}, we get
\begin{equation}\label{e4}
\begin{array}{l}
\displaystyle
\frac{1}{2}\frac{d}{dt}\int_1^2|y|^2dx+\frac{1}{2}\frac{d}{dt}cm\int_1^2\int_0^{\infty}\mu(s)|\gamma_x|^2dsdx+c(1-m)\int_1^2|y_x|^2dx-\frac{cm}{2}\int_1^2\int_0^{\infty}\mu^{\prime}(s)|\gamma_x|^2dsdx\\
\displaystyle
+\frac{cm}{2}\left[\int_1^2\mu(s)|\gamma_x|^2dx\right]_0^{\infty}+\Re\left(c(1-m)y_x(1)\overline{y}(1)\right)+\Re\left(cm\int_0^\infty\mu(s)\gamma_x(1,s)\overline{y}(1)ds\right)=0.
\end{array}
\end{equation}
Adding equations \eqref{e1} and \eqref{e4} and using the transmission conditions in \eqref{system1} and conditions \eqref{H},   we get
\begin{equation}
\begin{array}{l}
\displaystyle
\frac{1}{2}\frac{d}{dt} \left(\int_0^1 \frac{1}{\sigma}|u_t|^2dx+\int_0^1\eta |u_x|^2dx\right)+\frac{1}{2}\frac{d}{dt}\int_1^2|y|^2dx+\frac{1}{2}\frac{d}{dt}cm\int_1^2\int_0^{\infty}\mu(s)|\gamma_x|^2dsdx\\
\displaystyle
=
-c(1-m)\int_1^2|y_x|^2dx+\frac{cm}{2}\int_1^2\int_0^{\infty}\mu^{\prime}(s)|\gamma_x|^2dsdx.
\end{array}
\end{equation} 
Thus, the energy of the system \eqref{system1} can be written in the following form
\begin{equation}
\mathcal{E}(t)=\frac{1}{2}\int_0^1\left(\frac{1}{\sigma}|u_t|^2+\eta|u_x|^2\right)dx+\frac{1}{2}\int_1^2|y|^2dx+\frac{cm}{2}\int_1^2\int_0^{\infty}\mu(s)|\gamma_x|^2dsdx
\end{equation}
and
\begin{equation}
\frac{d}{dt}\mathcal{E}(t)=\frac{cm}{2}\int_1^2\int_0^{\infty}\mu^{\prime}(s)|\gamma_x|^2dsdx-c(1-m)\int_1^2|y_x|^2dx.
\end{equation}
Now, we define the Hilbert energy space by
$$\mathcal{H}=H^1_{\frac{1}{\sigma},L}(0,1)\times L^2_{\frac{1}{\sigma}}(0,1)\times L^2(1,2)\times W$$
equipped with the following inner product
$$\left<U_1,U_2\right>_{\mathcal{H}}=\int_0^1 \eta u_{1x}\overline{u_{2x}}dx+\int_0^1\frac{1}{\sigma}v_1\overline{v_2}dx+\int_1^2y_1\overline{y_2}dx+cm\int_1^2\int_0^{\infty}\mu(s)\gamma_{1x}\overline{\gamma_{2x}}dsdx,$$
where $U_i=(u_i, v_i, y_i,  \gamma_i)\in\mathcal{H}$, $i=1,2$. 
We denote by 
\begin{equation}\label{zeta_function}
\zeta=c(1-m)y+cm\int_0^{\infty}\mu(s)\gamma(s)ds.
\end{equation}
Defining the unbounded linear operator $\mathcal{A}$ by
$$\mathcal{A}(u, v, y, \gamma)^{\top}=(v, \sigma(\eta u_x)_x, \zeta_{xx}, -\gamma_s+y)^{\top};\quad \text{for all} \quad(u, v, y, \gamma)^{\top}\in D(\mathcal{A})$$
 and 
$$D(\mathcal{A})=\left\{
\begin{array}{lll}
(u, v, y, \gamma)\in\mathcal{H}; v\in H^1_{\frac{1}{\sigma},L}(0,1), \,  u\in H^2_{\frac{1}{\sigma},L}(0,1), \,  y\in H^1_R(1,2), \,  \zeta\in H^2(1,2), \vspace{0.2cm}\\
\gamma_s  \in W,  \quad\gamma(\cdot,0)=0, \quad \eta(1)u_x(1)=\zeta_x(1), \quad \text{and}\quad v(1)=y(1)
\end{array}\right\},
$$
we can rewrite \eqref{system1} as the following evolution equation
\begin{equation}\label{Cauchy}
U_t=\mathcal{A}U,\quad U(0)=U_0
\end{equation}
where $U_0=\left(u_0,v_0, y_0, \gamma_0\right)^{\top}$ with $\displaystyle\gamma_0=\int _0^s\varphi_0(x,r)dr$.\\
\begin{Proposition}\label{m-dissipative} Assume Hypotheses \ref{hyp1} and \ref{hyp3} and $m \in [0,1]$.  The unbounded linear operator $\mathcal{A}$ is m-dissipative in the energy space $\mathcal{H}$. 
\end{Proposition}
\begin{proof}
For all $U = (u, v, y, \gamma)^{\top}\in D(\mathcal{A})$, using condition \eqref{H}  and the fact that $m\in[0,1]$,  we get
$$\Re(\left<\mathcal{A}U,U\right>_{\mathcal{H}})=\frac{cm}{2}\int_1^2\int_0^{\infty}\mu^{\prime}(s)|\gamma_x|^2dsdx-c(1-m)\int_1^2|y_x|^2dx\leq 0,
$$
which implies that $\mathcal{A}$ is dissipative. Now, we need to prove that $\mathcal{A}$ is maximal.  For this aim, let $F =(f_1, f_2, f_3, f_4(\cdot,s))^{\top}\in \mathcal{H}$, we need to find $U = (u, v, y, \gamma)^{\top}\in D(\mathcal{A})$ unique solution of
\begin{equation}\label{maximal}
-\mathcal{A}U=F.
\end{equation}
Equivalently, we have the following system
\begin{eqnarray}
-v&=&f_1, \label{e11}\\
-\sigma(\eta u_x)_x&=&f_2, \label{e12}\\
-\zeta_{xx}&=&f_3, \label{e13}\\
\gamma_s-y &=&f_4(\cdot,s).\label{e14}
\end{eqnarray}
From \eqref{e11}, we have $v=-f_1$, thus $v\in H^1_{\frac{1}{\sigma}, L}(0,1)$.
From \eqref{e14}, we get
\begin{equation}\label{gamma}
\gamma(\cdot, s)=sy(x)+\int_0^s f_4(x,\tau)d\tau.
\end{equation}
Using \eqref{gamma} and the definition of $\zeta$, we get
\begin{equation}\label{zeta}
\zeta=c(1-m)y+cmy\int_0^{\infty} s\mu(s)ds+cm\int_0^{\infty}\mu(s)\int_0^s f_4(x,\tau)d\tau ds.
\end{equation}
Let  $\varphi_1 \in H^1_{\frac{1}{\sigma},L}(0,1)$ and $\varphi_2\in  H^1_R(1,2)$ such that $\varphi_1(1)=\varphi_2(1)$. Multiplying \eqref{e12} by $\displaystyle\frac{1}{\sigma}\overline{\varphi_1}$ and \eqref{e13} by $\overline{\varphi_2}$ and integrating by parts over $(0,1)$ and $(1,2)$, respectively, we get
\begin{equation}\label{add}
\int_0^1\eta u_x\overline{\varphi_{1x}}dx+\int_1^2\zeta_x\overline{\varphi_{2x}}dx=\int_0^1\frac{1}{\sigma} f_2\overline{\varphi_1}dx+\int_1^2f_3\overline{\varphi_2}dx.
\end{equation}
Using \eqref{zeta} and  \eqref{add}, we get
\begin{equation}\label{unique}
\mathcal{B}((u,\zeta),(\varphi_1,\varphi_2))=\mathcal{L}(\varphi_1,\varphi_2)\quad\text{for all}\quad (\varphi_1,\varphi_2)\in H^1_{\frac{1}{\sigma},L}(0,1)\times H^1_R(1,2),
\end{equation}
where 
\begin{equation*}
\mathcal{B}((u,\zeta),(\varphi_1,\varphi_2))=\int_0^1\eta u_x\overline{\varphi_{1x}}dx+\left(c(1-m)+cm\int_0^s s\mu(s)ds\right)\int_1^2 y_x\overline{\varphi_{2x}}dx,
\end{equation*}
and 
\begin{equation*}
\mathcal{L}(\varphi_1,\varphi_2)=\int_0^1\frac{1}{\sigma} f_2\overline{\varphi_1}dx+\int_1^2f_3\overline{\varphi_2}dx+cm\int_1^2\left(\int_0^{\infty}\mu(s)\int_0^sf_{4x}(x,\tau)d\tau ds\right) \overline{\varphi_{2x}}dx.
\end{equation*}
From \eqref{H},  we have $\displaystyle c(1-m)+cm\int_0^s s\mu(s)ds>0$; moreover, using Proposition \ref{prop_1} in the Appendix and the fact that $f_4\in W$, we get that $\displaystyle\int_0^sf_4(x,\tau)d\tau\in W$. Thus,  $\mathcal{B}$ is a sesquilinear, continuous and coercive form on $( H^1_{\frac{1}{\sigma},L}(0,1)\times H^1_R(1,2))^2$ and $\mathcal{L}$ is a linear and continuous form on  $H^1_{\frac{1}{\sigma},L}(0,1)\times H^1_R(1,2)$. using Lax-Milgram, \eqref{unique} admits a unique solution $(u,y)\in H^1_{\frac{1}{\sigma},L}(0,1)\times H^1_R(1,2)$. So,  using the fact that $f_2\in L^2_{\frac{1}{\sigma}}(0,1)$ we get $u\in H^2_{\frac{1}{\sigma},L}(0,1) $.  Since $f_4(\cdot,s)\in W$ and $y\in H^1_R(1,2)$, $\gamma_s(\cdot,s)\in W$. Now, in order to obtain that $\gamma(\cdot,s)\in W$, it is sufficient to prove that $\displaystyle \int_0^{\infty}\mu(s)\|\gamma_x(\cdot,s)\|^2_{L^2(1,2)}ds< \infty$. For this aim, let $\varepsilon_1, \varepsilon_2>0$, using Hypothesis \eqref{H}, we have
\begin{equation}\label{p1-}
\int_{\varepsilon_1}^{\varepsilon_2}\mu(s)\|\gamma_x(\cdot,s)\|^2_{L^2(1,2)}ds\leq \frac{-1}{K_{\mu}}\int_{\varepsilon_1}^{\varepsilon_2}\mu^{\prime}(s)\|\gamma_x(\cdot,s)\|^2_{L^2(1,2)}ds.
\end{equation}
Using integration by parts in \eqref{p1-}, we obtain
\begin{equation}\label{p2-}
\begin{array}{ll}
\displaystyle
\int_{\varepsilon_1}^{\varepsilon_2}\mu(s)\|\gamma_x(\cdot,s)\|^2_{L^2(1,2)}ds\leq \frac{1}{K_{\mu}}\int_{\varepsilon_1}^{\varepsilon_2}\mu(s)\frac{d}{ds}\left(\|\gamma_x(\cdot,s)\|^2_{L^2(1,2)}\right)ds\\
\displaystyle
+\frac{1}{K_{\mu}}\left(\mu(\varepsilon_1)\|\gamma_x(\cdot,\varepsilon_1)\|^2_{L^2(1,2)}-\mu(\varepsilon_2)\|\gamma_x(\cdot,\varepsilon_2)\|^2_{L^2(1,2)}\right).
\end{array}
\end{equation}
Now, from Young's inequality we have
\begin{equation*}
\begin{split}
 \frac{1}{K_{\mu}}\int_{\varepsilon_1}^{\varepsilon_2}\mu(s)\frac{d}{ds}\left(\|\gamma_x(\cdot,s)\|^2_{L^2(1,2)}\right)ds &=\frac{2}{K_{\mu}} \int_{\varepsilon_1}^{\varepsilon_2}\mu(s)\Re\left(\int_1^2\gamma_x(\cdot,s)\overline{\gamma_{sx}}(\cdot,s)dx\right)ds\\
& \leq \frac{1}{2}\int_{\varepsilon_1}^{\varepsilon_2}\mu(s)\|\gamma_x(\cdot,s)\|^2_{L^2(1,2)}ds+\frac{2}{K_{\mu}^2}\int_{\varepsilon_1}^{\varepsilon_2}\mu(s)\|\gamma_{sx}(\cdot,s)\|^2_{L^2(1,2)}ds.
\end{split}
\end{equation*}
Inserting the above inequality in \eqref{p2-}, we obtain
\begin{equation*}
\int_{\varepsilon_1}^{\varepsilon_2}\mu(s)\|\gamma_x(\cdot,s)\|^2_{L^2(1,2)}ds\leq \frac{4}{K_{\mu}^2}\int_{\varepsilon_1}^{\varepsilon_2}\mu(s)\|\gamma_{sx}(\cdot,s)\|^2_{L^2(1,2)}ds+\frac{2}{K_{\mu}}\mu(\varepsilon_1)\|\gamma_x(\cdot,\varepsilon_1)\|^2_{L^2(1,2)}-\frac{2}{K_{\mu}}\mu(\varepsilon_2)\|\gamma_x(\cdot,\varepsilon_2)\|^2_{L^2(1,2)}.
\end{equation*}
Taking the above inequality as $\varepsilon_1\to 0^{+}$ and $\varepsilon_2\to\infty$, and using the fact that $\gamma_s(\cdot,s)\in W$, $\gamma(\cdot,0)=0$ and condition \eqref{H}, we obtain
$$\int_0^{\infty}\mu(s)\|\gamma_x(\cdot,s)\|^2_{L^2(1,2)}ds< \infty,
$$
i.e., $\gamma(\cdot,s)\in W$. Hence,  $U\in D(\mathcal{A})$ and it is the unique solution of \eqref{maximal}. Then, $\mathcal{A}$ is an isomorphism. Moeover, using the fact that  $\rho(\mathcal{A})$ is open set of $\mathbb{C}$ (see Theorem 6.7 (Chapter III) in \cite{Kato01}), we easily get $R(\lambda I-\mathcal{A})=\mathcal{H}$ for a sufficiently small $\lambda>0$. This, together with the dissipativeness of $\mathcal{A}$, imply that $D(\mathcal{A})$ is dense in $\mathcal{H}$ and $\mathcal{A}$ is m-dissipative in $\mathcal{H}$ (see Theorem 4.5, 4.6 in \cite{Pazy01}). The proof is thus complete.  

\end{proof}
\section{\textbf{Strong Stability}}\label{Strong Stability}
The aim of this subsection is to prove the strong stability of \eqref{system1}.  First, we denote the following  
\begin{equation}\label{M}
M_1=\left\|x\left(\frac{a^{\prime}-b}{a}\right)\right\|_{L^{\infty}(0,1)}\quad  \text{and} \quad M_2=\left\|x\frac{b}{a}\right\|_{L^{\infty}(0,1)}.
\end{equation}

\begin{Hypothesis}\label{Condition1}
Assume Hypothesis \ref{hyp1}, $a$  \eqref{WD} or \eqref{SD} and the functions $a$ and $b$ such that 
$$\displaystyle M_1<1+\frac{K_a}{2}\quad\text{ and}\quad \displaystyle M_2<1-\frac{K_a}{2}.$$
\end{Hypothesis}
\begin{example}
Let $a(x)=x^{\mu_1}$ and $b(x)=c_bx^{\mu_2}$, such that $c_b\in\R^{\ast}$. For Hypothesis \ref{Condition1} to be attained, we need to have $\mu_1-\mu_2<1$ and $|c_b|<1-\frac{\mu_1}{2}$. 
\end{example}
\noindent The main result of this section is the following theorem.
\begin{Theorem}\label{Strong1}
Let $m\in [0,1]$ and assume condition \eqref{H} and Hypothesis \ref{Condition1}.  Then, the $C_0$-semigroup of contractions $(e^{t\mathcal{A}})_{t\geq 0}$ is strongly stable in $\mathcal{H}$, i.e., for all $U_0\in \mathcal{H}$, the solution of  \eqref{Cauchy} satisfies $\displaystyle\mathcal{E}(t)\xrightarrow[t\to\infty] {} 0$.
\end{Theorem}
\noindent According to Theorem of Arendt-Batty \cite{Arendt-Batty}, to prove Theorem \ref{Strong1}, we need to prove that the operator $\mathcal{A}$ has no pure imaginary eigenvalues and $\sigma(\mathcal{A})\cap i\R$ is countable. The proof of Theorem \ref{Strong1} will be established based on the following proposition.
\begin{Proposition}\label{proposition2}
Let $m \in [0, 1]$ and assume condition \eqref{H} and Hypothesis \ref{Condition1}, we have
\begin{equation}\label{pro1}
i\R\subset \rho(\mathcal{A}).
\end{equation}
\end{Proposition}
We will prove Proposition \ref{proposition2} by a contradiction argument. Remark that, it has been proved in Proposition \ref{m-dissipative} that $0\in\rho(\mathcal{A})$. Now, suppose that \eqref{pro1} is false, then there exists $\omega\in \R^{\ast}$ such that $i\omega\notin \rho(\mathcal{A})$. According to According to Remark A.1 in \cite{AkilCPAA},  page 25 in \cite{LiuZheng01},  and Remark A.3 in \cite{Akil2022}, there exists $\lbrace \la_n, U^n=(u^n, v^n, y^n, \gamma^n)^{\top}\rbrace_{n\geq 1}\subset \R^{\ast}\times D(\mathcal{A})$, such that
\begin{equation}\label{pro2}
\la_n\to \omega\quad \text{as}\quad \, n\to\infty\quad\text{and}\quad |\la_n|<|\omega|,
\end{equation}
\begin{equation}\label{pro3}
\|U^n\|_{\mathcal{H}}=\|(u^n, v^n, y^n, \gamma^n)^{\top}\|_{\mathcal{H}}=1,
\end{equation}
and
\begin{equation}\label{pro4}
(i\la_n I-\mathcal{A})U^n=F_n:=(f^1_n, f^2_n, f^3_n, f^4_n(\cdot,s))\to 0\quad\text{in}\quad\mathcal{H}, \quad\text{as}\quad n\to \infty.
\end{equation}
Detailing \eqref{pro4}, we get
\
\begin{eqnarray}
i\la_nu^n-v^n &=&f^1_n\quad \text{in}\quad H^1_{\frac{1}{\sigma},L}(0,1),\label{pro5}\\
i\la_nv^n-\sigma(\eta u_x^n)_x &=&f^2_n \quad \text{in}\quad L^2_{\frac{1}{\sigma}}(0,1),\label{pro6}\\
i\la_ny^n-\zeta_{xx}^n &=& f^3_n\quad \text{in}\quad L^2(1,2),\label{pro7}\\
i\la_n\gamma^n+\gamma_s^n-y^n &=& f^4_n(\cdot,s)\quad \text{in}\quad W\label{pro8}.
\end{eqnarray}
We will proof condition \eqref{pro1} by finding a contradiction with \eqref{pro3} such as $\|U^n\|_{\mathcal{H}}\to 0$.  The proof of this Proposition will rely on the forthcoming Lemmas.
\begin{Lemma}\label{lemma1}
Let $m\in[0,1]$ and assume condition \eqref{H} and Hypothesis \ref{Condition1}. Then, the solution $(u^n, v^n, y^n, \gamma^n)^{\top}\in D(\mathcal{A})$ of \eqref{pro5}-\eqref{pro8} satisfies
\begin{equation}\label{QQ1}
\left\{\begin{array}{ll}
\displaystyle
\int_1^2\int_0^{\infty} -\mu^{\prime}(s)|\gamma^n_x|^2dsdx\xrightarrow[n\to\infty] {} 0 \\
 \displaystyle
 \int_1^2\int_0^{\infty} \mu(s)|\gamma^n_x|^2dsdx\xrightarrow[n\to\infty] {} 0, \\
  \displaystyle
 \int_1^2|y^n_x|^2dx \xrightarrow[n\to\infty] {} 0\\ \displaystyle
 \int_1^2|y^n|^2dx \xrightarrow[n\to\infty] {} 0\\ \displaystyle
\text{and}  \int_1^2|\zeta^n_x|^2dx \xrightarrow[n\to\infty] {} 0.
\end{array}\right.
\end{equation}
\end{Lemma}

\begin{proof}
$\bullet$ Firstly, for the case where $m\in (0,1]$, we shall show the first and second limits in \eqref{QQ1}. Regarding the case in which $m=0$, these terms vanish.\\
Taking the inner product of \eqref{pro4} with $U$ in $\mathcal{H}$ and using the fact that$\|F_n\|_{\mathcal{H}}\to 0$ and $\|U^n\|_{\mathcal{H}}=1$ , we obtain
\begin{equation}\label{inner_product}
\frac{cm}{2}\int_1^2\int_0^{\infty}-\mu^{\prime}(s)|\gamma^n_x|^2dsdx+c(1-m)\int_1^2|y^n_x|^2dx=-\Re(\left<\mathcal{A}U,U\right>_{\mathcal{H}})\leq  \|F_n\|_{\mathcal{H}}\|U^n\|_{\mathcal{H}}\xrightarrow[n\to\infty] {} 0.
\end{equation}
Then,  using the fact that $m\in(0,1]$, we obtain the first limit in \eqref{QQ1}.
Using conditions \eqref{H},  we have
\begin{equation*}
\int_1^2\int_0^{\infty} \mu(s)|\gamma^n_x|^2dsdx\leq\frac{1}{K_{\mu}} \int_1^2\int_0^{\infty}- \mu^{\prime}(s)|\gamma^n_x|^2dsdx.
\end{equation*}
Then, using the above inequality and the first limit in \eqref{QQ1} we obtain the second limit in \eqref{QQ1}.\\
$\bullet$ Now, we will prove the third and the fourth limits in \eqref{QQ1} for the cases when $m\in[0,1)$ and $m=1$ separately.\\
\underline{\textbf{For the case when $\mathbf{m\in[0,1)}$}:}\\
From \eqref{inner_product} and the first limit in \eqref{QQ1},  we get the third limit in \eqref{QQ1}. Using the Poincar\'e inequality and the above result, we deduce the fourth limit in \eqref{QQ1}.\\
\underline{\textbf{For the case when $\mathbf{m=1}$}:}\\
Differentiate  \eqref{pro8} with respect to $x$,  then multiply by $\mu(s)\overline{y}_x$ and integrate over $(1,2)\times(0,\infty)$, we obtain
\begin{equation*}
g(0)\int_1^2|y^n_x|^2dx=i\la \int_1^2 \int_0^{\infty}\mu(s)\gamma^n_x\overline{y^n_x}dsdx+\int_1^2 \int_0^{\infty}\mu(s)\gamma^n_{sx}\overline{y^n_x}dsdx-\int_1^2 \int_0^{\infty}\mu(s)(f^4_n)_x\overline{y^n_x}dsdx.
\end{equation*}
Integrating by parts  the second term on the right hand side in the above equation with respect to $s$ and using the fact that $\gamma(\cdot,0)=0$ on $(1,2)$ and \eqref{H} , we get
\begin{equation}\label{g1}
g(0)\int_1^2|y^n_x|^2dx=i\la \int_1^2 \int_0^{\infty}\mu(s)\gamma^n_x\overline{y^n_x}dsdx+\int_1^2 \int_0^{\infty}-\mu^{\prime}(s)\gamma^n_{x}\overline{y^n_x}dsdx-\int_1^2 \int_0^{\infty}\mu(s)(f^4_n)_x\overline{y^n_x}dsdx.
\end{equation}
Using the Young and the Cauchy-Schwarz inequalities, we get
\begin{equation}\label{g2}
\left|i\la \int_1^2 \int_0^{\infty}\mu(s)\gamma^n_x\overline{y^n_x}dsdx\right|\leq 2|\la|^2 \int_1^2 \int_0^{\infty}\mu(s)|\gamma^n_x|^2dsdx+\frac{g(0)}{8}\int_1^2|y^n_x|^2dx,
\end{equation}
\begin{equation}\label{g3}
\begin{array}{ll}
\displaystyle
\left|\int_1^2 \int_0^{\infty}-\mu^{\prime}(s)\gamma^n_x\overline{y^n_x}dsdx\right|
&\displaystyle\leq \sqrt{\mu(0)} \left(\int_1^2 \int_0^{\infty}-\mu^{\prime}(s)|\gamma^n_x|^2dsdx\right)^{1/2}\left(\int_1^2|y^n_x|^2dsdx\right)^{1/2}
\\
&\displaystyle
\leq \frac{\mu(0)}{g(0)} \int_1^2 \int_0^{\infty}-\mu^{\prime}(s)|\gamma^n_x|^2dsdx+\frac{g(0)}{4}\int_1^2|y^n_x|^2dx,
\end{array}
\end{equation}
and 
\begin{equation}\label{g4}
\left|i\la \int_1^2 \int_0^{\infty}\mu(s) (f^4_n)_x\overline{y^n_x}dsdx\right|\leq 2 \int_1^2 \int_0^{\infty}\mu(s)|(f^4_n)_x|^2dsdx+\frac{g(0)}{8}\int_1^2|y^n_x|^2dx.
\end{equation}
So, using \eqref{g2}-\eqref{g4} in \eqref{g1}, the first and second limits in \eqref{QQ1},  we obtain the third limit in \eqref{QQ1}.  Using again the Poincar\'e inequality, we deduce the fourth limit in \eqref{QQ1}. \\
$\bullet$ Now, we will prove the last limit in \eqref{QQ1} when $m\in[0,1]$.\\
Applying the Young and the Cauchy-Schwarz inequalities, we obtain 
\begin{equation}\label{ineq4}
\begin{array}{ll}
\displaystyle
\int_1^2|\zeta^n_{x}|^2dx & \displaystyle\leq 2c^2(1-m)^2\int_1^2|y^n_x|^2dx+2c^2m^2\left(\int_0^{\infty}\mu(s)ds\right)\int_1^2\int_0^{\infty} \mu(s)|\gamma^n_x|^2dsdx\\ \displaystyle
&  \displaystyle \leq 2c^2(1-m)^2\int_1^2|y^n_x|^2dx+2c^2m^2 g(0)\int_1^2\int_0^{\infty} \mu(s)|\gamma^n_x|^2dsdx.
\end{array}
\end{equation}
Then, using the first and third limits in \eqref{QQ1} in the above inequality, we obtain the last limit in \eqref{QQ1}.  Thus the proof has been completed.
\end{proof}

\begin{Lemma}\label{lemma2}
Let $m\in[0,1]$, assume conditon \eqref{H} and Hypothesis \ref{Condition1}. Then, the solution $(u^n, v^n, y^n, \gamma^n)^{\top}\in D(\mathcal{A})$ of \eqref{pro5}-\eqref{pro8} satisfies
\begin{equation}\label{x1}
\int_0^1\eta|u_x^n|^2dx\xrightarrow[n\to\infty] {} 0\quad\text{and}\quad \int_0^1\frac{1}{\sigma}|v^n|^2dx \xrightarrow[n\to\infty] {} 0.
\end{equation}
\end{Lemma}

\begin{proof}
The proof of this Lemma is divided into three steps.\\
\textbf{Step 1.} The aim of this step is to show that 
\begin{equation}\label{x3}
|\la_n u^n(1)|\xrightarrow[n\to\infty] {} 0  \quad\text{and}\quad |u^n_x(1)|\xrightarrow[n\to\infty] {} 0.
\end{equation}
Thanks to Lemma \ref{lemma1}, we have that 
\begin{equation}\label{x4}
|y^n(1)|\leq\int_1^2|y^n_x|dx\leq \left(\int_1^2|y^n_x|^2dx\right)^{1/2}\xrightarrow[n\to\infty] {} 0.
\end{equation}
Now, using equation \eqref{pro7},  the Gagliardo-Nirenberg inequality, Lemma \ref{lemma1} and the fact that $\|U^n\|_{\mathcal{H}}=1$, $\|F_n\|_{\mathcal{H}}\to 0$,  and $|\la_n| \to |\omega|$, we obtain
\begin{equation}\label{x5}
|\zeta^n_x(1)|\leq c_1\|\zeta^n_{xx}\|^{1/2} \|\zeta^n_{x}\|^{1/2}+c_2 \|\zeta^n_{x}\|\leq c_1  \|\la _n y^n-f^3_n\|^{1/2}\|\zeta^n_{x}\|^{1/2}+c_2 \|\zeta^n_{x}\|\xrightarrow[n\to\infty] {} 0.
\end{equation}
Using \eqref{pro5}, the transmission conditions in \eqref{system1}, \eqref{x4} and the fact that 
\begin{equation}\label{f1n}
|f^1_n(1)|\leq \int_0^1|(f^1_n)_x|dx\leq \sqrt{\max_{x\in[0,1]}\eta^{-1}}\|\sqrt{\eta}(f^1_n)_x\|\leq \sqrt{\max_{x\in[0,1]}\eta^{-1}}\|F_n\|_{\mathcal{H}}\to 0,
\end{equation}
we obtain
$$|\la_n u^n(1)|\leq |v^n(1)|+|f^1_n(1)|\leq |y^n(1)|+|f^1_n(1)|\xrightarrow[n\to\infty] {} 0.$$
Now,  using the transmission conditions and \eqref{x5}, we get $\displaystyle|\eta(1)u^n_x(1)|=|\zeta^n_x(1)|\xrightarrow[n\to\infty] {} 0.$ Thus,  we get \eqref{x3}.\\
\textbf{Step 2.} The aim of this step is to prove the first limit in \eqref{x1}. Substitute $v^n=i\la_n u^n-f_n^1$ into \eqref{pro6}, we get
\begin{equation}\label{x8}
\la_n^2 u^n+\sigma(\eta u^n_x)_x=-i\la_n f_n^1-f^2_n.
\end{equation}
Multiplying \eqref{x8} by $\displaystyle -2\frac{x}{\sigma}\overline{u_x^n}+\frac{K_a}{2\sigma}\overline{u^n}$,   and integrating by parts over $(0,1)$, we obtain
\begin{equation*}
\begin{array}{l}\displaystyle
(1+\frac{K_a}{2})\int_0^1\frac{1}{\sigma}|\la_{n} u^n|^2dx+(1-\frac{K_a}{2})\int_0^1\eta| u^n_x|^2dx=\int_0^1\frac{x}{\sigma}\left(\frac{a^{\prime}-b}{a}\right)|\la_n u^2|^2dx+\int_0^1x\frac{b}{a}\eta|u_x^n|^2dx
\\ 
\displaystyle
+\eta(1) |u^n_x(1)|^2 
+\frac{1}{\sigma(1)}|\la_n u^n(1)|^2-\Re\left(\frac{K_a}{2}\eta(1) u_x^n(1)\overline{u^n}(1)\right)
+2\Re\left(\int_0^1\dfrac{x}{\sigma}f^2_n\overline{u^n_x}dx\right)\\\
\displaystyle
+2\Re\left(i\la_n\int_0^1\dfrac{x}{\sigma}f^1_n\overline{u^n_x}dx\right)
-\Re\left(i\frac{K_a}{2}\int_0^1\frac{1}{\sigma}f_n^1\la_n\overline{u^n}\right)
-\Re\left(\frac{K_a}{2}\int_0^1\frac{1}{\sigma}f_n^2\overline{u^n}\right).
\end{array}
\end{equation*}
Thus, we get
\begin{equation}\label{k1}
\begin{array}{l}\displaystyle
(1+\frac{K_a}{2}-M_1)\int_0^1\frac{1}{\sigma}|\la_{n} u^n|^2dx+(1-\frac{K_a}{2}-M_2)\int_0^1\eta| u^n_x|^2dx\leq\eta(1) |u^n_x(1)|^2 \\ 
\displaystyle

+\frac{1}{\sigma(1)}|\la_n u^n(1)|^2+\left|\Re\left(\frac{K_a}{2}\eta(1) u_x^n(1)\overline{u^n}(1)\right)\right|
+\left|2\Re\left(\int_0^1\dfrac{x}{\sigma}f^2_n\overline{u^n_x}dx\right)\right|\\
\displaystyle
+\left|2\Re\left(i\la_n\int_0^1\dfrac{x}{\sigma}f^1_n\overline{u^n_x}dx\right)\right|
+\left|\Re\left(i\frac{K_a}{2}\int_0^1\frac{1}{\sigma}f_n^1\la_n\overline{u^n}\right)\right|
+\left|\Re\left(\frac{K_a}{2}\int_0^1\frac{1}{\sigma}f_n^2\overline{u^n}\right)\right|.
\end{array}
\end{equation}
Now, we consider the last four terms in \eqref{k1}. First, using the Cauchy-Schwarz inequality,  the fact that $\|F_n\|\to 0$,  $\|U^n\|=1$,  and the monotonicity of $\dfrac{x}{\sqrt{a}}$ in $(0,1]$, we get
\begin{equation}\label{x10}
\left|2\Re\left(\int_0^1\dfrac{x}{\sigma}f^2_n\overline{u^n_x}dx\right)\right|\leq 2\int_0^1\frac{x}{\sqrt{a}}\frac{1}{\sqrt{\sigma}}|f^2_n|\sqrt{\eta}|u^n_x|dx\leq \frac{2}{\sqrt{a(1)}}\|F_n\|_{\mathcal{H}}\|U^n\|_{\mathcal{H}}\xrightarrow[n\to\infty] {} 0.
\end{equation}
Applying the Cauchy-Schwarz and the Hardy-Poincaré inequalities, using the fact that $\|F_n\|\to 0$,  $\|U^n\|=1$,  $|\la_n| \to |\omega|$ and again the monotonicity of $\dfrac{x}{\sqrt{a}}$, we have
\begin{equation}\label{x11}
\left|2\Re\left(i\la_n\int_0^1\dfrac{x}{\sigma}f^1_n\overline{u^n_x}dx\right)\right| \leq 2|\la_n|\int_0^1 \frac{x}{\sqrt{a}}\frac{1}{\sqrt{\sigma}}|f^1_n|\sqrt{\eta}|u^n_x|dx \leq 2|\la_n|\sqrt{\frac{c_0}{a(1)}}\|F_n\|_{\mathcal{H}}\|U^n\|_{\mathcal{H}}\xrightarrow[n\to\infty] {} 0,
\end{equation}
\begin{equation}\label{x11new}
\Re\left(i\frac{K_a}{2}\int_0^1\frac{1}{\sigma}f_n^1\la_n\overline{u^n}\right)\leq |\la_n|\frac{K_a}{2}c_0^2\|F_n\|_{\mathcal{H}}\|U^n\|_{\mathcal{H}}\xrightarrow[n\to\infty] {} 0,
\end{equation}
and
\begin{equation}\label{x112new}
\Re\left(\frac{K_a}{2}\int_0^1\frac{1}{\sigma}f_n^2\overline{u^n}\right)\leq \frac{K_a}{2}c_0\|F_n\|_{\mathcal{H}}\|U^n\|_{\mathcal{H}}\xrightarrow[n\to\infty] {} 0.
\end{equation}
where $c_0=\sqrt{C_{HP}\max\limits_{x\in[0,1]}\eta^{-1}}$.\\
Now, using \eqref{x3} and \eqref{x10}-\eqref{x112new} in \eqref{k1} and the fact that $\la_n\to \omega$, we get
\begin{equation}\label{x12}
(1+\frac{K_a}{2}-M_1)\int_0^1 \frac{1}{\sigma}|\la_n u^n|^2dx+(1-\frac{K_a}{2}-M_2)\int_0^1\eta |u^n_x|^2dx\xrightarrow[n\to\infty] {} 0.
\end{equation}
Thus, using Hypothesis \ref{Condition1} in the above equation we get the desired result.\\
\textbf{Step 4.} The aim of this step is to prove the second limit in \eqref{x1}. From \eqref{pro5},  and using \eqref{x12},  Hardy-Poincare inequality given in Proposition \ref{Hardy} in the Appendix and the fact that $\|F_n\|\to 0$,  we get
\begin{equation}
\int_0^1\frac{1}{\sigma}|v^n|^2dx\leq 2\int_0^1\frac{1}{\sigma}|\la_n u^n|^2dx+2C_{HP}\max_{x\in [0,1]}\eta\int_0^1|(f^1_n)_x|^2dx\xrightarrow[n\to\infty] {} 0.
\end{equation}
\end{proof}

\textbf{Proof of Proposition \ref{proposition2}} From Lemmas \ref{lemma1} and \ref{lemma2}, we obtain that $\|U^n\|_{\mathcal{H}}\to 0$ as $n\to 0$ which contradicts that $\|U^n\|_{\mathcal{H}}=1$ in \eqref{pro3}. Then, \eqref{pro1} holds true and the proof is complete.

\section{\textbf{Polynomial stability in the case of Coleman-Gurtin Heat conduction law }}\label{SectionPoly}
This section is devoted to study the polynomial stability of the system under consideration, specifically in the case of the Coleman-Gurtin heat conduction law with the parameter $m$ belonging to the interval $(0,1)$. The main result of this section is presented in the following theorem.
\begin{Theorem}\label{polynomial Stab}
Assume condition \eqref{H}, Hypothesis \ref{Condition1} and $m\in(0,1)$. Then, there exits $C_0>0$ such that for all $U_0\in D(\mathcal{A})$ we have
\begin{equation}
\mathcal{E}(t)\leq \frac{C_0}{t^{4}}\|U_0\|_{D(\mathcal{A})}^2,\quad t>0.
\end{equation}
\end{Theorem}
\noindent According to Theorem of Borichev and Tomilov \cite{Borichev01} (see also \cite{RaoLiu01} and \cite{Batty01}), in order to prove Theorem \ref{polynomial Stab} we need to prove that the following two conditions hold:
\begin{equation}\label{pol1}\tag{R1}
i\R\subset \rho(\mathcal{A}),
\end{equation}
\begin{equation}\label{pol2}\tag{R2}
\limsup_{|\la|\to\infty}\frac{1}{\la^{\frac{1}{2}}}\|(i\la I-\mathcal{A})^{-1}\|<\infty.
\end{equation}
\begin{Proposition}\label{proposition3}
Under condition \eqref{H} and Hypothesis \ref{Condition1}, let  $( \la, U=(u, v, y, \gamma)^{\top})\subset \R^{\ast}\times D(\mathcal{A})$, with $|\la|\geq 1$ such that
\begin{equation}\label{pol3}
(i\la I-\mathcal{A})U=F:=(f^1, f^2, f^3, f^4(\cdot,s))^{\top}\in\mathcal{H},
\end{equation}
i.e.
\begin{eqnarray}
i\la u-v&=&f^1,\label{pol4}\\
i\la v-\sigma(\eta u_x)_x &=&f^2,\label{pol5}\\
i\la y -\zeta_{xx}  &=& f^3,\label{pol6}\\
i\la \gamma+\gamma_s-y &=& f^4(\cdot,s) \label{pol7}.
\end{eqnarray}
Then, we have the following inequality
\begin{equation}\label{pol8}
\|U\|_{\mathcal{H}}\leq K_1(1+|\la|^{\frac{1}{2}}) \|F\|_{\mathcal{H}},
\end{equation}
where $K_1$ is a constant independent of $\la$ to be determined.
\end{Proposition}
In order to prove Proposition \ref{proposition3}, we need the following Lemmas.
\begin{Lemma}\label{lem1_pol}
Assume condition \eqref{H}, Hypothesis \ref{Condition1},  $m\in(0,1)$ and $|\la|\geq 1$. Then, the solution $(u, v, y, \gamma)^{\top}\in D(\mathcal{A})$ of \eqref{pol3} satisfies the following estimates
\begin{equation}\label{est0}
\int_1^2\int_0^{\infty}-\mu^{\prime}(s)|\gamma_x|^2dsdx\leq \kappa_0  \left(\|U\|_{\mathcal{H}}\|F\|_{\mathcal{H}}+\|F\|^2_{\mathcal{H}}\right),
\end{equation} 
\begin{equation}\label{est1}
\int_1^2 |y_x|^2dx\leq \kappa_1  \left(\|U\|_{\mathcal{H}}\|F\|_{\mathcal{H}}+\|F\|^2_{\mathcal{H}}\right),
\end{equation}
\begin{equation}\label{est2}
\int_1^2\int_0^{\infty}\mu(s)|\gamma_x|^2dsdx\leq \kappa_2  \left(\|U\|_{\mathcal{H}}\|F\|_{\mathcal{H}}+\|F\|^2_{\mathcal{H}}\right)
\end{equation}
and
\begin{equation}\label{est3}
\int_1^2|\zeta_x|^2dx\leq \kappa_3   \left(\|U\|_{\mathcal{H}}\|F\|_{\mathcal{H}}+\|F\|^2_{\mathcal{H}}\right),
\end{equation}
where  $\displaystyle\kappa_0=\frac{2}{cm}$, $\displaystyle\kappa_1=\frac{1}{c(1-m)}$, $\displaystyle\kappa_2=\frac{2}{cmK_{\mu}}$,  and $\displaystyle\kappa_3=2c(1-m)+\frac{4cmg(0)}{K_{\mu}}$.
\end{Lemma}
\begin{proof}
First taking the inner product of \eqref{pol3} with $U$ in $\mathcal{H}$, we get
\begin{equation}\label{p1}
\frac{cm}{2}\int_1^2\int_0^{\infty}-\mu^{\prime}(s)|\gamma_x|^2dsdx+c(1-m)\int_1^2|y_x|^2dx=-\Re(\left<\mathcal{A}U,U\right>_{\mathcal{H}})\leq  \|F\|_{\mathcal{H}}\|U\|_{\mathcal{H}}.
\end{equation}
From the above equation, using condition \eqref{H} and the fact that  $m\in(0,1)$, we obtain \eqref{est0},  \eqref{est1} and 
\begin{equation*}
\int_1^2\int_0^{\infty} \mu(s)|\gamma_x|^2dsdx\leq\frac{1}{K_{\mu}} \int_1^2\int_0^{\infty}- \mu^{\prime}(s)|\gamma_x|^2dsdx.
\end{equation*}
Then, using the above inequality and \eqref{est0},  we obtain \eqref{est2}.  Now, using  the Cauchy-Schwarz inequality, we get
\begin{equation*}
\begin{array}{ll}
\displaystyle
\int_1^2|\zeta_x|^2dx &\leq  \displaystyle 2c^2(1-m)^2\int_1^2|y_x|^2dx+2c^2m^2\left(\int_0^{\infty}\mu(s)ds\right)\int_1^2\int_0^{\infty}\mu(s)|\gamma_x|^2dsdx\\
 &\leq\displaystyle  2c^2(1-m)^2\int_1^2|y_x|^2dx+2c^2m^2g(0)\int_1^2\int_0^{\infty}\mu(s)|\gamma_x|^2dsdx.
\end{array}
\end{equation*}
Using \eqref{est1} and \eqref{est2} in the above inequality we obtain \eqref{est3} and the proof is complete.
\end{proof}
\begin{Lemma}\label{lem2_pol}
Assume condition \eqref{H}, Hypothesis \ref{Condition1},  $m\in(0,1)$ and $|\la|\geq 1$. Then, the solution $(u, v, y, \gamma)^{\top}\in D(\mathcal{A})$ of \eqref{pol3} satisfies the following estimates
\begin{equation}\label{est4}
\int_1^2|y|^2dx\leq \frac{ \kappa_4}{|\la|}  \left(\|U\|_{\mathcal{H}}\|F\|_{\mathcal{H}}+\|F\|^2_{\mathcal{H}}\right),
\end{equation}
where $\displaystyle\kappa_4$ is a constant independent of $\la$ to be determined below.
\end{Lemma}
\begin{proof}
Multiplying equation \eqref{pol6} by $\displaystyle\frac{1}{i\la}\overline{y}$, integrating over $(1,2)$ and taking the real part, we have
\begin{equation}\label{p3}
\int_1^2|y|^2dx+\Re\left(\frac{1}{i\la} \int_1^2\zeta_x\overline{y}_xdx\right)+\Re\left[\frac{1}{i\la}\zeta_x(1)\overline{y}(1)\right]=\Re\left(\frac{1}{i\la}\int_1^2 f^3\overline{y}dx\right).
\end{equation}
Thanks to \eqref{est1} and \eqref{est3},  we get
\begin{equation}\label{p4}
\left|\Re\left(\frac{1}{i\la} \int_1^2\zeta_x\overline{y}_xdx\right)\right| \leq \frac{\kappa_5}{|\la|} \left(\|U\|_{\mathcal{H}}\|F\|_{\mathcal{H}}+\|F\|^2_{\mathcal{H}}\right),
\end{equation}
and 
\begin{equation}\label{p5}
\left|\Re\left(\frac{1}{i\la}\int_1^2 f^3\overline{y}dx\right)\right| \leq \frac{1}{|\la|} \left(\|U\|_{\mathcal{H}}\|F\|_{\mathcal{H}}+\|F\|^2_{\mathcal{H}}\right),
\end{equation}
where $\displaystyle\kappa_5=\sqrt{\kappa_1\kappa_3}$.
Now, using the Gagliardo-Nirenberg inequality we can estimate the term $\displaystyle\Re\left[\frac{1}{i\la}\zeta_x(1)\overline{y}(1)\right]$ in the following way
\begin{equation}\label{p6}
\begin{array}{ll}
\displaystyle
|\Re\left[\frac{1}{i\la}\zeta_x(1)\overline{y}(1)\right]|\leq \frac{1}{|\la|}\left(\alpha_1\|\zeta_{xx}\|^{\frac{1}{2}}\|\zeta_{x}\|^{\frac{1}{2}}+\alpha_2\|\zeta_{x}\|\right)\left(\alpha_3\|y_{x}\|^{\frac{1}{2}}\|y\|^{\frac{1}{2}}+\alpha_4\|y\|\right)\\
\displaystyle
\leq\frac{\beta}{|\la|}\left(\|\zeta_{xx}\|^{\frac{1}{2}}\|\zeta_{x}\|^{\frac{1}{2}}+\|\zeta_{x}\|\right)\left(\|y_{x}\|^{\frac{1}{2}}\|y\|^{\frac{1}{2}}+\|y\|\right),
\end{array}
\end{equation} 
where $\displaystyle\beta=\max(\alpha_1,\alpha_2)\max(\alpha_3,\alpha_4)$, $\alpha_i>0$, $i=1,2,3,4$. 
In order to estimate the terms in \eqref{p6}, we will use \eqref{pol6}, the fact that  $|\la|\geq 1$ and the inequality $\displaystyle x_1x_2\leq \epsilon x_1^2+\frac{1}{4\epsilon}x_2^2$ for $\epsilon>0$. Thus, we obtain

\begin{equation*}\label{p7}
 \begin{array}{ll}
 \displaystyle \bullet\,
 \frac{\beta}{|\la|}\|\zeta_{xx}\|^{\frac{1}{2}}\|\zeta_{x}\|^{\frac{1}{2}}\|y\|\leq \frac{\beta}{|\la|^{1/2}}\|y\|^{\frac{1}{2}} \|\zeta_{x}\|^{\frac{1}{2}} \|y\|+\frac{\beta}{|\la|} \|f^3\|^{\frac{1}{2}} \|\zeta_{x}\|^{\frac{1}{2}} \|y\| \leq \frac{19\beta^2}{4|\la|}\|\zeta_x\|  \|y\|+ \frac{1}{19} \|y\|^2
 \\
  \displaystyle
  +\frac{19\beta^2}{4|\la|^2}  \|f^3\|  \|\zeta_x\| +\frac{1}{19} \|y\|^2
 \leq 
\frac{19^3\beta^4}{16|\la|^2}  \|\zeta_x\|^2+\frac{1}{76} \|y\|^2+\frac{2}{19} \|y\|^2+\frac{19^2\beta^4}{8|\la|^2}\|f^3\|^2+\frac{1}{2|\la|^2} \|\zeta_x\|^2
\\
  \displaystyle
\leq 
\frac{9}{76} \|y\|^2+\frac{19^2\beta^4}{8|\la|^2}\|f^3\|^2+\left(\frac{1}{2}+\frac{19^3\beta^4}{16}\right)\frac{1}{|\la^2|}\|\zeta_x\|^2
 \leq  \frac{9}{76} \|y\|^2+ \frac{\kappa_6}{|\la|}\left(\|U\|_{\mathcal{H}}\|F\|_{\mathcal{H}}+\|F\|^2_{\mathcal{H}}\right),\\
\\
 
 \displaystyle\bullet\,
 \frac{\beta}{|\la|}\|\zeta_{xx}\|^{\frac{1}{2}}\|\zeta_{x}\|^{\frac{1}{2}} \|y_x\|^{\frac{1}{2}}\|y\|^{\frac{1}{2}}\leq \frac{\beta}{|\la|^{1/2}}\|y\| \|\zeta_{x}\|^{\frac{1}{2}} \|y_x\|^{\frac{1}{2}}+\frac{\beta}{|\la|} \|f^3\|^{\frac{1}{2}} \|\zeta_{x}\|^{\frac{1}{2}} \|y_x\|^{\frac{1}{2}}\|y\|^{\frac{1}{2}} 
\\ \displaystyle
\leq 
\frac{1}{19} \|y\|^2+\frac{19\beta^2}{4|\la|} \|\zeta_{x}\|  \|y_x\|+ \frac{\beta^2}{2|\la|}\|f^3\| \|y\|+ \frac{1}{2|\la|} \|\zeta_{x}\|  \|y_{x}\|\leq \frac{1}{19}\|y\|^2+\frac{\kappa_7}{|\la|}\left(\|U\|_{\mathcal{H}}\|F\|_{\mathcal{H}}+\|F\|^2_{\mathcal{H}}\right),\\
\\
 \displaystyle\bullet\,
\frac{\beta}{|\la|}\|\zeta_x\|\|y\|\leq \frac{19\beta^2}{4|\la|^2} \|\zeta_x\|^2+\frac{1}{19}\|y\|^2
\leq \frac{1}{19}\|y\|^2+ \frac{\kappa_8}{|\la|} \left(\|U\|_{\mathcal{H}}\|F\|_{\mathcal{H}}+\|F\|^2_{\mathcal{H}}\right),\\
\\
 \displaystyle\bullet\,
 \frac{\beta}{|\la|} \|\zeta_x\| \|y_x\|^{\frac{1}{2}}\|y\|^{\frac{1}{2}}\leq \frac{\beta^2}{2|\la|} \|\zeta_x\|^2+\frac{1}{2|\la|}\|y_x\| \|y\|\leq \frac{\beta^2}{2|\la|}\|\zeta_x\|^2+\frac{19}{8|\la|^2}\|y_x\|^2+\frac{1}{38}\|y\|^2\\
 \displaystyle
 \leq \frac{1}{38}\|y\|^2+\frac{\kappa_9}{|\la|} \left(\|U\|_{\mathcal{H}}\|F\|_{\mathcal{H}}+\|F\|^2_{\mathcal{H}}\right),
 \end{array}
 \end{equation*}
 where $ \displaystyle\kappa_6=\left(\frac{1}{2}+\frac{19^3\beta^4}{16}\right)\kappa_3+\frac{19^2\beta^4}{8}$,  $ \displaystyle\kappa_7=\left(\frac{19\beta^2}{4}+\frac{1}{2}\right)\sqrt{\kappa_1\kappa_3}+\frac{\beta^2}{2}$,  $\displaystyle\kappa_8= \frac{19\beta^2\kappa_3}{4} $ and $\displaystyle\kappa_9= \frac{\beta^2\kappa_3}{2}+\frac{19\kappa_1}{8}$.\\
Inserting the above inequalities in \eqref{p6}, we obtain
 \begin{equation*}\label{p11}
\left |\Re\left[\frac{1}{i\la}\zeta_x(1)\overline{y}(1)\right]\right|\leq  \frac{1}{4}\|y\|^2+ \frac{\kappa_{10}}{|\la|}\left(\|U\|_{\mathcal{H}}\|F\|_{\mathcal{H}}+\|F\|^2_{\mathcal{H}}\right),
 \end{equation*}
 where $ \displaystyle \kappa_{10}=\sum_{{i=6}}^9\kappa_i $.
 Finally, using the above inequality,  \eqref{p4} and \eqref{p5}  in \eqref{p3}, we obtain \eqref{est4} with $\displaystyle\kappa_4=\frac{4(\kappa_{10}+\kappa_5+1)}{3}$.
\end{proof}
\begin{Lemma}\label{lem3_pol}
Assume condition \eqref{H} , Hypothesis \ref{Condition1},  $m\in(0,1)$ and $|\la|\geq 1$. Then, the solution $(u, v, y, \gamma)^{\top}\in D(\mathcal{A})$ of \eqref{pol3} satisfies the following estimates
\begin{equation}\label{est5}
|\zeta_x(1)|^2\leq \kappa_{11}(1+|\la|^{\frac{1}{2}}) \left(\|U\|_{\mathcal{H}}\|F\|_{\mathcal{H}}+\|F\|^2_{\mathcal{H}}\right),
\end{equation}
and
\begin{equation}\label{est6}
|y(1)|^2\leq \kappa_{12}(1+|\la|^{{\frac{1}{2}}}) \left(\|U\|_{\mathcal{H}}\|F\|_{\mathcal{H}}+\|F\|^2_{\mathcal{H}}\right),
\end{equation}
where $\kappa_{11},  \kappa_{12}$ are constants independent of $\la$ to be determined.
\end{Lemma}
\begin{proof}
First, multiplying equation \eqref{pol6} by $2(x-2)\overline{\zeta}_x$, integrating over $(1,2)$ and taking the real part, we obtain
\begin{equation}\label{p12}
|\zeta_x(1)|^2=\Re\left(2i\la\int_1^2(x-2)\overline{\zeta}_xydx\right)+\int_1^2|\zeta_x|^2dx-\Re\left(2\int_1^2(x-2)f^3\overline{\zeta}_xdx \right).
\end{equation}
Thanks to \eqref{est3} and Lemma \ref{lem2_pol}, the first and last terms on the right hand side in \eqref{p12} can be estimated in the following way
\begin{equation*}\label{p13}
\begin{array}{ll}
\displaystyle\left|\Re\left(2i\la\int_1^2(x-2)\overline{\zeta}_xydx\right)\right|\leq \kappa_{13}|\la|^{\frac{1}{2}} \left(\|U\|_{\mathcal{H}}\|F\|_{\mathcal{H}}+\|F\|^2_{\mathcal{H}}\right)
\end{array}
\end{equation*}
and
\begin{equation*}\label{p14}
\begin{array}{ll}
\displaystyle
|\Re\left(2\int_1^2(x-2)f^3\overline{\zeta}_xdx \right)|  \leq   2\kappa_3^{\frac{1}{2}} \|F\|_{\mathcal{H}} (\|U\|_{\mathcal{H}}\|F\|_{\mathcal{H}}+\|F\|^2_{\mathcal{H}})^{1/2}
\leq \kappa_{14} \left(\|U\|_{\mathcal{H}}\|F\|_{\mathcal{H}}+\|F\|^2_{\mathcal{H}}\right),
\end{array}
\end{equation*}
where $\displaystyle\kappa_{13}=2\sqrt{\kappa_3\kappa_4}$ and $\displaystyle\kappa_{14}=3\sqrt{\kappa_3}$.\\
Thus, using the above two inequalities in  \eqref{p12}, we obtain
\begin{equation}
|\zeta_x(1)|^2\leq  \left(\kappa_3+\kappa_{14}+\kappa_{13}|\la|^{\frac{1}{2}}\right)\left(\|U\|_{\mathcal{H}}\|F\|_{\mathcal{H}}+\|F\|^2_{\mathcal{H}}\right)
\end{equation}
and   \eqref{est5} follows with $\displaystyle\kappa_{11}=\max\{\kappa_3+\kappa_{14},\kappa_{13}\}$.
Now, differentiating equation \eqref{pol7} with respect to $x$,  multiplying by $\mu(s)\overline{y}$ and integrating over $(1,2)\times(0,\infty)$, we obtain
\begin{equation}\label{p15}
\begin{array}{ll}
\displaystyle
\Re\left(i\la\int_1^2\int_0^{\infty}\mu(s)\gamma_x\overline{y}dsdx\right)+\Re\left(\int_1^2\int_0^{\infty}\mu(s)\gamma_{sx}\overline{y}dsdx\right)-g(0)\Re\left(\int_1^2y_x\overline{y}dx\right)\\
\displaystyle
=\Re\left(\int_1^2\int_0^{\infty}\mu(s)f^4_x\overline{y}dsdx\right).
\end{array}
\end{equation}
Integrating the second term in the above equation with respect to $s$ and using condition \eqref{H}, we obtain
\begin{equation}\label{p16}
\begin{array}{l}\displaystyle
\frac{g(0)}{2}|y(1)|^2
\displaystyle
=\Re\left(\int_1^2\int_0^{\infty}\mu(s)f^4_x\overline{y}dsdx\right)-\Re\left(\int_1^2\int_0^{\infty}(i\la\mu(s)-\mu^{\prime}(s))\gamma_x\overline{y}dsdx\right).\end{array}
\end{equation}
Thanks to Lemmas \ref{lem1_pol} and \ref{lem2_pol} and the fact that $|\la|\geq 1$ and $f^4\in W$, we can estimate the terms on the right hand side  in \eqref{p16} in the following way
\begin{equation*}\label{p17}
\begin{array}{ll}
\displaystyle
\left|\Re\left(i\la\int_1^2\int_0^{\infty}\mu(s)\gamma_x\overline{y}dsdx\right)\right| &\displaystyle \leq \sqrt{g(0)} |\la|\left(\int_1^2\int_0^{\infty}\mu(s)|\gamma_x|^2dsdx\right)^{1/2}\left(\int_1^2|y|^2dx\right)^{1/2}\\
&\displaystyle
\leq \kappa_{15}|\la|^{\frac{1}{2}}\left(\|U\|_{\mathcal{H}}\|F\|_{\mathcal{H}}+\|F\|^2_{\mathcal{H}}\right),
\end{array}
\end{equation*}
\begin{equation*}\label{p18}
\begin{array}{ll}
\displaystyle
\left|\Re\left(\int_1^2\int_0^{\infty}\mu^{\prime}(s)\gamma_{x}\overline{y}dsdx\right)\right| &\displaystyle \leq \sqrt{\mu(0)}\left(\int_1^2\int_0^{\infty}-\mu^{\prime}(s)|\gamma_x|^2dsdx\right)^{1/2}\left(\int_1^2|y|^2dx\right)^{1/2}\\[0.15 in]
&\displaystyle
\leq \kappa_{16}\left(\|U\|_{\mathcal{H}}\|F\|_{\mathcal{H}}+\|F\|^2_{\mathcal{H}}\right),
\end{array}
\end{equation*}
and
\begin{equation*}\label{p19}
\begin{array}{ll}
\displaystyle
\left|\Re\left(\int_1^2\int_0^{\infty}\mu(s)f^4_x\overline{y}dsdx\right)\right| &\displaystyle \leq \sqrt{g(0)} \left(\int_1^2\int_0^{\infty}\mu(s)|f^4_x|^2dsdx\right)^{1/2}\left(\int_1^2|y|^2dx\right)^{1/2}\\[0.15 in]
&\displaystyle
\leq \sqrt{\frac{g(0)}{cm}} \left(\|U\|_{\mathcal{H}}\|F\|_{\mathcal{H}}+\|F\|^2_{\mathcal{H}}\right),
\end{array}
\end{equation*}
where $\displaystyle\kappa_{15}=\sqrt{g(0)\kappa_2\kappa_4}$ and $\displaystyle\kappa_{16}=\sqrt{\mu(0)\kappa_0\kappa_4}$. 
Thus, using the above inequalities in \eqref{p16}, we obtain \eqref{est5} with 
$\displaystyle \kappa_{12}=\frac{2}{g(0)}\max\{\kappa_{15}, \kappa_{16}+\sqrt{g(0)c^{-1}m^{-1}}\}$ and the proof of the Lemma is complete.
\end{proof}

\noindent Now, substituting equation \eqref{pol4} into \eqref{pol5}, we get
\begin{equation}\label{pol9}
\la^2 u+\sigma(\eta u_x)_x=-(i\la f^1+f^2).
\end{equation}
Using equation \eqref{pol4} and the Hardy-Ponicar\'e inequality given in Proposition \ref{Hardy} (see the Appendix),  we get
\begin{equation}\label{lambdau}
\|\la u\|_{\frac{1}{\sigma}}\leq \|v\|_{\frac{1}{\sigma}}+\sqrt{C_{HP}}\|f^1_x\|
\leq \max\{1, c_0\}\left(\|v\|_{\frac{1}{\sigma}}+\|\sqrt{\eta}f_x^1\|\right)
\leq c_1 \left[\|U\|_{\mathcal{H}}+\|F\|_{\mathcal{H}}\right],
\end{equation}
where $c_1=\max\{1,c_0\}$ with $c_0=\sqrt{C_{HP}\max\limits_{x\in[0,1]}\eta^{-1}(x)}$.\\
\begin{Lemma}\label{lem4_pol}
Assume condition \eqref{H},  Hypothesis \ref{Condition1},  $m\in(0,1)$ and $|\la|\geq 1$. Then, the solution $(u, v, y, \gamma)^{\top}\in D(\mathcal{A})$ of \eqref{pol3} are such that
\begin{equation}\label{est7}
\int_0^1 \eta|u_x|^2dx\leq  \kappa _{17} (1+|\la|^{\frac{1}{2}})\left(\|U\|_{\mathcal{H}}\|F\|_{\mathcal{H}}+\|F\|^2_{\mathcal{H}}\right)
\end{equation}
and
\begin{equation}\label{est8}
\int_0^1 \frac{1}{\sigma}|v|^2dx\leq   \kappa _{18} (1+|\la|^{\frac{1}{2}})\left(\|U\|_{\mathcal{H}}\|F\|_{\mathcal{H}}+\|F\|^2_{\mathcal{H}}\right),
\end{equation}
where $\kappa_{17}$ and $\kappa_{18}$ are constants independent of $\la$ to be determined below.
\end{Lemma}
\begin{proof} For simplicity, we will divide the proof of this result into several steps. First of all ,we observe that thanks to  Lemma \ref{lemmaAppendix} in the Appendix, we have that equation \eqref{p20new} holds.  \\
\underline{\textbf{Step 1.}} The aim of the first step is to estimate the terms in \eqref{p20new}. \\
$\bullet$ Using the Cauchy-Schwarz inequality, the monotonicity of the function $\dfrac{x}{\sqrt{a}}$ in $(0,1]$ (being $a$ \eqref{WD} or \eqref{SD})  and the fact that $K_a<2$,  we can estimate the term $\displaystyle  2\Re\left(\int_0^1f^2\dfrac{x}{\sigma}\overline{u_x}dx\right)$ as:
\begin{equation}\label{p25}
\left|2\Re\left(\int_0^{1}\dfrac{x}{\sigma}f^2\overline{u}_xdx\right)\right|\leq 2\int_{0}^{1}\frac{x}{\sqrt{a}}\frac{1}{\sqrt{\sigma}}|f^2| \sqrt{\eta}|{u}_x|dx\leq \frac{2}{\sqrt{a(1)}}\left(\|U\|_{\mathcal{H}}\|F\|_{\mathcal{H}}+\|F\|_{\mathcal{H}}^2\right).
\end{equation}
$\bullet$ Now, consider the term $\displaystyle  2\Re\left(i\int_0^1 \left(\frac{xf^1}{\sigma}\right)_x\la \bar{u}dx\right)$.  We have
\begin{equation}\label{p26}
2\Re\left(i\int_0^1 \left(\frac{xf^1}{\sigma}\right)_x\la \overline{u}dx\right)=2\Re\left(i\int_0^1 \frac{x }{\sigma}f^1_x\la \overline{u}dx\right)+2\Re\left(i\int_0^1 \left(\frac{x}{\sigma}\right)^{\prime}f^1\la \overline{u}dx\right).
\end{equation}
Using again the monotonicity of $\dfrac{x}{\sqrt{a}}$ and \eqref{lambdau}, we obtain
\begin{equation}\label{p27}
\left|2\Re\left(i\int_0^{1} \frac{x}{\sigma}f^1_x\la \bar{u}dx\right)\right| 
\leq 2\int_0^{1} \frac{x}{\sqrt{a}} \sqrt{\eta}|f^1_x|\frac{1}{\sqrt{\sigma}}|\la u|dx\leq  \frac{2c_1}{\sqrt{a(1)}} \left(\|U\|_{\mathcal{H}}\|F\|_{\mathcal{H}}+\|F\|_{\mathcal{H}}^2\right).
\end{equation}
Moreover, to the second term in \eqref{p26} can be rewritten as
\begin{equation}\label{p28}
2\Re\left(i\int_0^1 \left(\frac{x}{\sigma}\right)^{\prime}f^1\la \bar{u}dx\right)=2\Re\left(i\int_0^1 \dfrac{1}{\sigma} f^1\la \bar{u}dx\right)+2\Re\left(i\int_0^1 \dfrac{x}{\sigma}\left(\dfrac{a^{\prime}-b}{a}\right) f^1\la \bar{u}dx\right).
\end{equation}
Using \eqref{lambdau} and the Hardy-Poincar\'e inequality,  the first term in the right hand side of \eqref{p28} can be estimated as
 \begin{equation}\label{p29}
\left|2\Re\left(i\int_0^1 \dfrac{1}{\sigma} f^1\la \bar{u}dx\right)\right|\leq 2\int_0^1 \frac{1}{\sqrt{\sigma}}|f^1| \frac{1}{\sqrt{\sigma}}|\la \bar{u}|dx\leq 2c_0c_1 \left(\|U\|_{\mathcal{H}}\|F\|_{\mathcal{H}}+\|F\|_{\mathcal{H}}^2\right).
\end{equation}
Now,  using Hypothesis \ref{Condition1}, we have the following estimate for the second term in the right hand side of \eqref{p28}:
\begin{equation}\label{p30}
\begin{array}{l}
\displaystyle
\left| 2\Re\left(i\int_0^1 \dfrac{ x}{\sigma}\left(\dfrac{a^{\prime}-b}{a}\right) f^1\la \bar{u}dx\right)\right| \leq 2M_1 \|f^1\|_{\frac{1}{\sigma}}\|\la u\|_{\frac{1}{\sigma}}\leq  c_0c_1(2+K_a)\left(\|U\|_{\mathcal{H}}\|F\|_{\mathcal{H}}+\|F\|_{\mathcal{H}}^2\right).
\end{array}
\end{equation}
Then, inserting \eqref{p27}, \eqref{p29},  and \eqref{p30} in \eqref{p26} we can conclude that
\begin{equation}\label{p31}
\left|2\Re\left(i\int_0^1 \left(\frac{xf^1}{\sigma}\right)_x\la \bar{u}dx\right)\right|\leq \left(\frac{2c_1}{\sqrt{a(1)}}+c_0c_1(4+K_a)\right) \left(\|U\|_{\mathcal{H}}\|F\|_{\mathcal{H}}+\|F\|_{\mathcal{H}}^2\right).
\end{equation}
$\bullet$ Now, using \eqref{lambdau} and the Cauchy-Schwarz and the Hardy-Poinacr\'e inequalities, we get 
\begin{equation}\label{p31*}
\frac{K_a}{2}\Re\left(i\int_0^1\frac{1}{\sigma}f^1\la \overline{u}dx\right)\leq \frac{K_a}{2}c_0c_1\left(\|U\|_{\mathcal{H}}\|F\|_{\mathcal{H}}+\|F\|_{\mathcal{H}}^2\right)
\end{equation}
and
\begin{equation}\label{p31**}
\frac{K_a}{2}\Re\left(\int_0^1\frac{1}{\sigma}f^2 \overline{u}dx\right)\leq \frac{K_a}{2}c_1\left(\|U\|_{\mathcal{H}}\|F\|_{\mathcal{H}}+\|F\|_{\mathcal{H}}^2\right).
\end{equation}
Thus, using equations \eqref{p25} and \eqref{p31}-\eqref{p31**} in \eqref{p20new}, we obtain
\begin{equation}\label{result1}
\begin{array}{l}
\displaystyle\left(1+\frac{K_a}{2}-M_1\right)\int_0^1 \frac{1}{\sigma}|\la u|^2dx+\left(1-\frac{K_a}{2}- M_2\right)\int_0^1 \eta |u_x|^2dx
\leq\kappa_{19} \left(\|U\|_{\mathcal{H}}\|F\|_{\mathcal{H}}+\|F\|_{\mathcal{H}}^2\right)\\
\displaystyle+\frac{1}{\sigma(1)}|\la u(1)|^2
+\eta(1)|u_x(1)|^2
+\frac{2}{\sigma(1)}|f^1(1)| |\la u(1)|+\frac{K_a}{2}\eta(1)|u_x(1)| |u(1)|,
\end{array}
\end{equation}
where $\displaystyle\kappa_{19}=\frac{2(c_1+1)}{\sqrt{a(1)}}+c_0c_1(4+K_a)+\frac{K}{2}c_0c_1+\frac{K}{2}c_1$.\\
\underline{\textbf{Step 2.}} The aim of this step is to show \eqref{est7}.
Using the transmission conditions and Lemma \ref{lem3_pol}, we obtain
\begin{equation}\label{p32}
\eta(1)|u_x(1)|^2\leq \frac{\kappa_{11}}{\eta(1)}(1+|\la|^{\frac{1}{2}}) \left(\|U\|_{\mathcal{H}}\|F\|_{\mathcal{H}}+\|F\|^2_{\mathcal{H}}\right)
\end{equation}
and 
\begin{equation}\label{p33}
|v(1)|^2\leq \kappa_{12}(1+|\la|^{{\frac{1}{2}}}) \left(\|U\|_{\mathcal{H}}\|F\|_{\mathcal{H}}+\|F\|^2_{\mathcal{H}}\right).
\end{equation}
Using equation \eqref{pol4}, \eqref{p33} and the fact that $\displaystyle|f^1(1)|\leq \sqrt{\max_{x\in[0,1]}\eta^{-1}}\|F\|_{\mathcal{H}}$, we obtain 
\begin{equation}\label{p34}
|\la u(1)|^2\leq 2|v(1)|^2+2|f^1(1)|^2\leq \kappa_{20} (1+|\la|^{{\frac{1}{2}}}) \left(\|U\|_{\mathcal{H}}\|F\|_{\mathcal{H}}+\|F\|^2_{\mathcal{H}}\right),
\end{equation}
where $\displaystyle\kappa_{20}=2(\kappa_{12}+\max_{x\in[0,1]}\eta^{-1})$.
Now, using Young's inequality we get
\begin{equation}\label{p35}
\frac{2}{\sigma(1)}|f^1(1)| |\la u(1)|\leq \frac{1}{\sigma(1)}|f^1(1)|^2+\frac{1}{\sigma(1)}|\la u(1)|^2\leq \kappa_{21} (1+|\la|^{\frac{1}{2}}) \left(\|U\|_{\mathcal{H}}\|F\|_{\mathcal{H}}+\|F\|^2_{\mathcal{H}}\right)
\end{equation}
and 
\begin{equation}\label{p35*}
\frac{K_a}{2}\eta(1)|u_x(1)| |u(1)|\leq \tilde{\kappa}_{21}(1+|\la|^{\frac{1}{2}}) \left(\|U\|_{\mathcal{H}}\|F\|_{\mathcal{H}}+\|F\|^2_{\mathcal{H}}\right),
\end{equation}

where $\displaystyle\kappa_{21}=\frac{1}{\sigma(1)}(\kappa_{20}+\max_{x\in[0,1]}\eta^{-1})$ and $\tilde{\kappa}_{21}=\frac{K^2}{8}\kappa_{11}+\frac{\kappa_{20}}{2}$.
Finally, substituting equations \eqref{p32}, \eqref{p34}-\eqref{p35*} into \eqref{result1}, we obtain
\begin{equation}\label{result2}
\begin{array}{l}
\displaystyle\left(1+\frac{K}{2}-M_1\right)\int_0^1 \frac{1}{\sigma}|\la u|^2dx+\left(1-\frac{K}{2}-M_2\right)\int_0^1 \eta |u_x|^2dx\\
\displaystyle
\leq\kappa_{22}(1+|\la|^{{\frac{1}{2}}})  \left(\|U\|_{\mathcal{H}}\|F\|_{\mathcal{H}}+\|F\|_{\mathcal{H}}^2\right),
\end{array}
\end{equation}
with $\displaystyle\kappa_{22}=\kappa_{19}+\frac{\kappa_{11}}{\eta(1)}+\frac{\kappa_{20}}{\sigma(1)}+\kappa_{21}+\tilde{\kappa}_{21}$.
Therefore, using Hypothesis \ref{Condition1} in \eqref{result2}, we obtain \eqref{est7} with $\kappa_{17}=\frac{\kappa_{22}}{1-\frac{K}{2}-M_2}$.\\
\underline{\textbf{Step 3.}} The aim of this step is to show \eqref{est8}. From inequality \eqref{result2}, we deduce 
\begin{equation*}
\int_0^1 \frac{1}{\sigma} |\la u|^2dx\leq \frac{\kappa_{22}}{1+\frac{K}{2}-M_1}(1+|\la|^{{\frac{1}{2}}})  \left(\|U\|_{\mathcal{H}}\|F\|_{\mathcal{H}}+\|F\|_{\mathcal{H}}^2\right).
\end{equation*}
Then,  using \eqref{pol4} and the above inequality, we get
\begin{equation*}
\int_0^1 \frac{1}{\sigma} |v|^2dx\leq 2\frac{\kappa_{22}}{1+\frac{K}{2}-M_1} (1+|\la|^{{\frac{1}{2}}})  \left(\|U\|_{\mathcal{H}}\|F\|_{\mathcal{H}}+\|F\|_{\mathcal{H}}^2\right)+2c_0^2\|F\|^2_{\mathcal{H}}.
\end{equation*}
Therefore, we reach \eqref{est8} with $\displaystyle\kappa_{18}=\frac{2\kappa_{22}}{1+\frac{K}{2}-M_1}+2c_0^2$.\\
\end{proof}

\noindent\textbf{Proof of Proposition \ref{proposition3}}.
Adding estimations \eqref{est2}, \eqref{est4}, \eqref{est7} and \eqref{est8} and using the fact that $|\la|\geq 1$, we obtain
\begin{equation}
\|U\|^2_{\mathcal{H}}\leq \kappa_{23} (1+|\la|^{{\frac{1}{2}}})  \left(\|U\|_{\mathcal{H}}\|F\|_{\mathcal{H}}+\|F\|_{\mathcal{H}}^2\right),
\end{equation}
where $\displaystyle\kappa_{23}=cm\kappa_{2}+\kappa_{4}+\kappa_{17}+\kappa_{18}$.
Thanks to Young's inequality, we have
\begin{equation*}
\|U\|^2_{\mathcal{H}}\leq \left( \kappa_{23}(1+|\la|^{{\frac{1}{2}}}) ^2+2\kappa_{23}(1+|\la|^{{\frac{1}{2}}})\right) \|F\|_{\mathcal{H}}^2,
\end{equation*}
and using the fact that $ \left( \kappa_{23}(1+|\la|^{{\frac{1}{2}}}) ^2+2\kappa_{23}(1+|\la|^{{\frac{1}{2}}})\right)\leq \left(\kappa_{23}(1+|\la|^{{\frac{1}{2}}})+1\right)^2
$, we obtain
\begin{equation}
\|U\|^2_{\mathcal{H}}\leq (\kappa_{23}+1)^2(1+|\la|^{{\frac{1}{2}}})^2 \|F\|_{\mathcal{H}}^2,
\end{equation}
Thus, $\|U\|_{\mathcal{H}}\leq K_1 \|F\|_{\mathcal{H}}$ with $K_1=\kappa_{23}+1$ and the proof of the Proposition is thus completed.\\
\textbf{Proof of Theorem \ref{polynomial Stab}}.  From Proposition \ref{proposition2}, we get \eqref{pol1}.  Next, we will prove \eqref{pol2}  by a contradiction argument. Suppose that  there exists
$( \la_n, U_n=(u^n, v,^n y^n, \gamma^n)^{\top})\subset \R^{\ast}\times D(\mathcal{A})$, with $|\la_n|\geq 1$ such that  $|\la _n|\to\infty$, $\|U_n\|_{\mathcal{H}}=1$ and there exists a sequence $F_n=(f^1_n,f_n^2, f_n^3, f_n^4)\in\mathcal{H}$ such that
$$\left(i\la_nI-\mathcal{A}\right)U_n=F_n\to 0\quad \text{in}\quad \mathcal{H}\quad\text{as}\quad n\to 0.$$
From Proposition \ref{proposition3} and taking $U=U_n$, $F=\la^{-\frac{1}{2}}F_n$ and $\la=\la_n$,  we can deduce that  $\|U_n\|_{\mathcal{H}}\to 0$, when  $|\la _n|\to\infty$, which contradicts $\|U_n\|_{\mathcal{H}}=1$. Thus, condition \eqref{pol2} holds true. The result follows from Huang-Pr\"uss Theorem (see  \cite{Huang01} and  \cite{pruss01}) and the proof is thus completed.

\begin{Remark}
For the case $m=0$, by proceeding with the same calculations as in the above section we reach same energy decay rate. 
\end{Remark}

\section{\textbf{Exponential stability in the case of Gurtin-Pipkin Heat conduction law }}\label{SectionExp}
\noindent In this section we will study the exponential stability for the system, specifically in the case of the Gurtin-Pipkin heat conduction law with the parameter $m=1$. The following theorem gives the main result of this section.
\begin{Theorem}\label{exponential_Stab}
Assume conditions \eqref{H} and Hypothesis \ref{Condition1}. Then,  the $C_0-$semigroup of contractions $\left(e^{t\mathcal{A}}\right)_{t\geq 0}$ is exponentially  stable, i.e. there exist constants $C_1\geq 1$ and $\tau_1>0$ independent of $U_0$ such that 
$$
\left\|e^{t\mathcal{A}}U_0\right\|_{\mathcal{H}}\leq C_1e^{-\tau_1 t}\|U_0\|_{\mathcal{H}},\qquad t\geq 0.
$$
\end{Theorem}
\noindent According to Huang \cite{Huang01} and Pruss \cite{pruss01}, we have to check if the following conditions hold:
\begin{equation}\label{E1}\tag{${\rm E1}$}
 i\mathbb{R}\subseteq \rho\left(\mathcal{A}\right)
\end{equation}
and 
\begin{equation}\label{E2}\tag{${\rm E2}$}
\displaystyle{\sup_{\la\in \R}\|\left(i\la I-\mathcal{A}\right)^{-1}\|_{\mathcal{L}\left(\mathcal{H}\right)}=O(1).}
\end{equation}
The following proposition is a technical finding that will be used to prove Theorem \ref{exponential_Stab}.

\begin{Proposition}\label{proposition4}
Assume condition \ref{H} and Hypothesis \ref{Condition1} and let $(\lambda,U:=(u,v))\in \R^{\ast}\times D(\mathcal{A})$, with $|\lambda|\geq1$, such that
\begin{equation}\label{t0}
(i\la I-\mathcal{A})U=F:=(f^1, f^2, f^3, f^4(\cdot,s))\in\mathcal{H},
\end{equation}
i.e. 
\begin{eqnarray}
i\la u-v&=&f^1,\label{t1}\\
i\la v-\sigma(\eta u_x)_x &=&f^2 ,\label{t2}\\
i\la y -\zeta_{xx}  &=& f^3,\label{t3}\\
i\la \gamma+\gamma_s-y &=& f^4(\cdot,s) \label{t4}.
\end{eqnarray}
Then,  we have the following inequality
\begin{equation}\label{t5}
\|U\|_{\mathcal{H}}\leq \ibt{K_2}  \|F\|_{\mathcal{H}},
\end{equation}
where $K_2$ is a suitable positive constant independent of $\lambda$ to be determined below.
Here we note that in this case we have 
\begin{equation}\label{zeta m=1}
\zeta=c\int_0^{\infty}\mu(s)\gamma(s)ds.
\end{equation}
\end{Proposition}
Proceeding as in Lemma \ref{lem1_pol}, we have the next result.
\begin{Lemma}\label{lem1-exp}
Assume condition \eqref{H}, Hypothesis  \ref{Condition1} hold,  $m=1$ and $\ibt{|\la|\geq 1}$. Then, the solution $(u, v, y, \gamma)^{\top}\in D(\mathcal{A})$ of \eqref{t0} satisfies the following estimates
\begin{equation}\label{est0_exp}
\int_1^2\int_0^{\infty}-\mu^{\prime}(s)|\gamma_x|^2dsdx\leq e_0  \left(\|U\|_{\mathcal{H}}\|F\|_{\mathcal{H}}+\|F\|^2_{\mathcal{H}}\right),
\end{equation} 
\begin{equation}\label{est1_exp}
\int_1^2\int_0^{\infty}\mu(s)|\gamma_x|^2dsdx\leq e_1  \left(\|U\|_{\mathcal{H}}\|F\|_{\mathcal{H}}+\|F\|^2_{\mathcal{H}}\right)
\end{equation}
and
\begin{equation}\label{est2_exp}
\int_1^2|\zeta_x|^2dx\leq e_2   \left(\|U\|_{\mathcal{H}}\|F\|_{\mathcal{H}}+\|F\|^2_{\mathcal{H}}\right),
\end{equation}
where  $\displaystyle e_0=\frac{2}{c}$, $\displaystyle e_1=\frac{2}{cK_{\mu}}$,  and $\displaystyle e_2=\frac{2cg(0)}{K_{\mu}}$.
\end{Lemma}

\begin{Lemma}\label{lem2-exp}
Assume condition \eqref{H}, Hypothesis \ref{Condition1},  $m=1$ and $\ibt{|\la|\geq 1}$. Then, the solution $(u, v, y, \gamma)^{top}\in D(\mathcal{A})$ of \eqref{t0} satisfies \begin{equation}\label{est3_exp}
\int_1^2|y_x|^2dx\leq e_3 |\la|^2  \left(\|U\|_{\mathcal{H}}\|F\|_{\mathcal{H}}+\|F\|^2_{\mathcal{H}}\right),
\end{equation}
where $e_3$ is a constant independent of $\la$ to be determined below.
\end{Lemma}
\begin{proof}
Solving \eqref{t4} we have
\begin{equation}\label{solution1}
\gamma=\left(\frac{1-e^{-i\la s}}{i\la}\right)y+\mathcal{F}(s), \quad\text{where}\quad \mathcal{F}(s)=\int_0^s e^{-i\la(s-\tau)}f^4(\tau)d\tau.
\end{equation}
Differentiating the above equation with respect to $x$, multiplying by $\mu(s)$ and integrating over $(1,2)\times (0,\infty)$, we have
\begin{equation}
\mathcal{N}\frac{1}{|\la|^2}\int_1^2|y_x|^2dx\leq 2\int_1^2\int_0^{\infty}\mu(s)|\gamma_x|^2dsdx+2\int_1^2\int_0^{\infty}\mu(s)|\mathcal{F}_x(s)|^2dsdx,
\end{equation}
where $\displaystyle\mathcal{N}=\int_0^{\infty}\mu(s)|1-e^{-i\la s}|^2ds$. 
Observe that, since $f^4\in W$, $\|\mathcal{F}\|_{W}\leq m_1 \|f^4\|_W$ with $m_1>0$ (see Proposition \ref{prop_1} in the Appendix); moreover, $\displaystyle\inf_{\la\in\R;|\la|\geq \varepsilon>0}\mathcal{N}\geq m_2>0$ (see Proposition \ref{prop_2} in the Appendix). These two inequalities and
\eqref{est1_exp} imply
\begin{equation}
\frac{m_2}{|\la|^2}\int_1^2|y_x|^2dx\leq 2(e_1+m_1c^{-1}) \left(\|U\|_{\mathcal{H}}\|F\|_{\mathcal{H}}+\|F\|^2_{\mathcal{H}}\right).
\end{equation}
Hence \eqref{est3_exp} holds with $e_3=\frac{2}{m_2} (e_1+m_1c^{-1})$.
\end{proof}

\begin{Lemma}\label{lem3-exp}
Assume condition \eqref{H}, Hypothesis \ref{Condition1},  $m=1$ and $\ibt{|\la|\geq 1}$. Then, the solution $(u, v, y, \gamma)^{\top}\in D(\mathcal{A})$ of \eqref{t0} satisfies the following inequality
\begin{equation}\label{est4_exp}
\int_1^2|y|^2dx\leq e_4  \left(\|U\|_{\mathcal{H}}\|F\|_{\mathcal{H}}+\|F\|^2_{\mathcal{H}}\right),
\end{equation}
where $e_4$ is a constant independent of $\la$ to be determined below.
\end{Lemma}
\begin{proof}
Multiplying \eqref{t3} by $\displaystyle\frac{1}{i\la}\overline{y}$, integrating over $(1,2)$ and taking the real part, we obtain
\begin{equation}\label{t6}
\int_1^2|y|^2dx=-\Re\left(\frac{1}{i\la} \int_1^2\zeta_x\overline{y}_xdx\right)-\Re\left[\frac{1}{i\la}\zeta_x(1)\overline{y}(1)\right]+\Re\left(\frac{1}{i\la}\int_1^2 f^3\overline{y}dx\right).
\end{equation}
Using Lemmas \ref{lem1-exp} and \ref{lem2-exp}, and the fact that $\ibt{|\la|\geq 1}$,  we obtain
\begin{equation}\label{t6-1}
\Re\left(\frac{1}{i\la} \int_1^2\zeta_x\overline{y}_xdx\right)\leq \sqrt{e_2e_3}\left(\|U\|_{\mathcal{H}}\|F\|_{\mathcal{H}}+\|F\|^2_{\mathcal{H}}\right)
\end{equation}
and
\begin{equation}\label{t6-2}
\Re\left(\frac{1}{i\la}\int_1^2 f^3\overline{y}dx\right)\leq \left(\|U\|_{\mathcal{H}}\|F\|_{\mathcal{H}}+\|F\|^2_{\mathcal{H}}\right),
\end{equation}
Now, proceeding as in  Lemma \ref{lem2_pol} for the term $\displaystyle\Re\left[\frac{1}{i\la}\zeta_x(1)\overline{y}(1)\right]$, using the Gagliardo-Nirenberg inequality, Lemmas \ref{lem1-exp} and \eqref{lem2-exp} and  the fact that $|\la|\geq 1$, we obtain
\begin{equation*}\label{t7}
 \begin{array}{ll}
 \displaystyle
 \frac{\beta}{|\la|}\|\zeta_{xx}\|^{\frac{1}{2}}\|\zeta_{x}\|^{\frac{1}{2}}\|y\|
 \leq  \frac{9}{76} \|y\|^2+ e_5\left(\|U\|_{\mathcal{H}}\|F\|_{\mathcal{H}}+\|F\|^2_{\mathcal{H}}\right),
 \end{array}
 \end{equation*}

 \begin{equation*}\label{t8}
 \begin{array}{ll}
 \displaystyle
 \frac{\beta}{|\la|}\|\zeta_{xx}\|^{\frac{1}{2}}\|\zeta_{x}\|^{\frac{1}{2}} \|y_x\|^{\frac{1}{2}}\|y\|^{\frac{1}{2}}\leq \frac{1}{19}\|y\|^2+(e_6+\frac{\beta^2}{2})\left(\|U\|_{\mathcal{H}}\|F\|_{\mathcal{H}}+\|F\|^2_{\mathcal{H}}\right),
 \end{array}
 \end{equation*}
\begin{equation*}\label{t9}
\frac{\beta}{|\la|}\|\zeta_x\|\|y\|\leq \frac{1}{19}\|y\|^2+ e_7 \left(\|U\|_{\mathcal{H}}\|F\|_{\mathcal{H}}+\|F\|^2_{\mathcal{H}}\right)
\end{equation*}
and
 \begin{equation*}\label{t10}
 \begin{array}{ll} \displaystyle
 \frac{\beta}{|\la|} \|\zeta_x\| \|y_x\|^{\frac{1}{2}}\|y\|^{\frac{1}{2}}
 \leq \frac{1}{38}\|y\|^2+\left(\frac{19e_3}{8}+\frac{\beta^2e_2}{2}\right)\left(\|U\|_{\mathcal{H}}\|F\|_{\mathcal{H}}+\|F\|^2_{\mathcal{H}}\right),
 \end{array}
 \end{equation*}
 where $ \displaystyle e_5=\left(\frac{1}{2}+\frac{19^3\beta^4}{16}\right)e_2+\frac{19^2\beta^4}{8}$,  $ \displaystyle e_6=\left(\frac{19\beta^2}{4}+\frac{1}{2}\right)\sqrt{e_2e_3}$,  and $\displaystyle e_7= \frac{19\beta^2\kappa_3}{4} $.
Then,  using the above inequalities, we obtain
\begin{equation}\label{t11}
\left|\Re\left[\frac{1}{i\la}\zeta_x(1)\overline{y}(1)\right]\right|\leq \frac{1}{4}\|y\|^2+ e_8\left(\|U\|_{\mathcal{H}}\|F\|_{\mathcal{H}}+\|F\|^2_{\mathcal{H}}\right),
\end{equation}
where $\displaystyle e_8=e_5+e_6+e_7+\frac{\beta^2(e_2+1)}{2}+\frac{19e_3}{8}$.
Thus, using \eqref{t6-1}, \eqref{t6-2} and \eqref{t11} in \eqref{t6}, we get the desired result with $\displaystyle e_4=\frac{4}{3}(\sqrt{e_2e_3}+e_8+1)$.
\end{proof}

\begin{Lemma}\label{lem4-exp}
Assume condition \eqref{H}, Hypothesis \ref{Condition1},  $m=1$ and $\ibt{|\la|\geq 1}$. Then, the solution $(u, v, y, \gamma)^{\top}\in D(\mathcal{A})$ of \eqref{t0} satisfies 
\begin{equation}\label{est5-exp}
|\zeta_x(1)|^2\leq e_{9} \left(\|U\|_{\mathcal{H}}\|F\|_{\mathcal{H}}+\|F\|^2_{\mathcal{H}}\right)
\end{equation}
and
\begin{equation}\label{est6-exp}
|y(1)|^2\leq e_{10} \left(\|U\|_{\mathcal{H}}\|F\|_{\mathcal{H}}+\|F\|^2_{\mathcal{H}}\right),
\end{equation}
where $e_9$ and $e_{10}$ are constants independent of $\la$ to be determined below.
\end{Lemma}
\begin{proof}
Multiplying \eqref{t3} by $2(x-2)\overline{\zeta}_x$, integrating over $(1,2)$ and taking the real part, we obtain
\begin{equation}\label{t12}
|\zeta_x(1)|^2=\Re\left(2i\la\int_1^2(x-2)\overline{\zeta}_xydx\right)+\int_1^2|\zeta_x|^2dx-\Re\left(2\int_1^2(x-2)f^3\overline{\zeta}_xdx \right).
\end{equation}
From \eqref{solution1},  we have 
\begin{equation*}
i\la \gamma=(1-e^{-i\la s})y+i\la \mathcal{F}(s).
\end{equation*}
Differentiating the above equation with respect to $x$, multiplying by $c\mu(s)$, integrating over $(0,\infty)$  and using definition of $\zeta$ in \eqref{zeta m=1} we obtain
\begin{equation*}\label{t13}
i\la \overline{\zeta}_x=-\mathcal{Q}\overline{y}_x+i\la c \int_0^{\infty}\mu(s)\overline{\mathcal{F}_x}(s)ds,
\end{equation*}
where $\displaystyle\mathcal{Q}=cg(0)-c\int_0^{\infty}\mu(s)e^{i\la s}ds$.
Multiplying  the above equation  by $2(x-2)\overline{y}$, integrating over $(1,2)$ and taking the real part, we have
\begin{equation*}
\begin{array}{ll}
\displaystyle
\Re\left(2i\la\int_1^2(x-2)\overline{\zeta}_xydx\right)=\Re\left(-\mathcal{Q}\int_1^22(x-2)y\overline{y}_xdx\right)+\Re\left(2i\la c\int_1^2(x-2)y\int_0^{\infty}\mu(s)\overline{\mathcal{F}}_x(s)dsdx\right)\\
\displaystyle
=cg(0)\int_1^2 |y|^2dx-cg(0)|y(1)|^2+\Re\left(2c\int_0^{\infty}\mu(s)e^{i\la s}ds\int_1^2(x-2)y\overline{y}_xdx\right)\\
\displaystyle
+\Re\left(2i\la c\int_1^2(x-2)y\int_0^{\infty}\mu(s)\overline{\mathcal{F}}_x(s)dsdx\right).
\end{array}
\end{equation*}
Using the above equation in \eqref{t12}, one has
\begin{equation}\label{t14}
\begin{array}{ll}
\displaystyle
cg(0)|y(1)|^2+|\zeta_x(1)|^2=cg(0)\int_1^2 |y|^2dx+\int_1^2|\zeta_x|^2dx+\Re\left(2c\int_0^{\infty}\mu(s)e^{i\la s}ds\int_1^2(x-2)y\overline{y}_xdx\right)\\
\displaystyle
+\Re\left(2i\la c\int_1^2(x-2)y\int_0^{\infty}\mu(s)\overline{\mathcal{F}}_x(s)dsdx\right)-
\Re\left(2\int_1^2(x-2)f^3\overline{\zeta}_xdx \right).
\end{array}
\end{equation}
Thanks to Lemma \ref{lem1-exp}, it results
\begin{equation}\label{t15}
\left|\Re\left(2\int_1^2(x-2)f^3\overline{\zeta}_xdx \right)\right|\leq 3e_2^{\frac{1}{2}}  \left(\|U\|_{\mathcal{H}}\|F\|_{\mathcal{H}}+\|F\|^2_{\mathcal{H}}\right).
\end{equation}
Integrating by parts we obtain
\begin{equation*}
\begin{array}{ll}
\displaystyle
\int_0^{\infty}\mu(s)e^{i\la s}ds =\frac{1}{i\la}\int_0^{\infty}-\mu^{\prime}(s)e^{i\la s}ds+\lim_{s\to \infty}\frac{1}{i\la}\mu(s)e^{i\la s}-\frac{\mu(0)}{i\la}=\frac{1}{i\la}\int_0^{\infty}-\mu^{\prime}(s)e^{i\la s}ds-\frac{\mu(0)}{i\la};
\end{array}
\end{equation*}
hence $$\left|\int_0^{\infty}\mu(s)e^{i\la s}ds\right| \leq \frac{2\mu(0)}{|\la|}.$$
Then, using the above inequality and Lemmas \ref{lem2-exp} and \ref{lem3-exp},  we obtain
\begin{equation}\label{t16}
\begin{array}{ll}
\displaystyle
\left|\Re\left(2c\int_0^{\infty}\mu(s)e^{i\la s}ds\int_1^2(x-2)y\overline{y}_xdx\right)\right|\leq \frac{4c\mu(0)}{|\la|}\|y\|\|y_x\|\leq 4c\mu(0)\sqrt{e_3e_4}\left(\|U\|_{\mathcal{H}}\|F\|_{\mathcal{H}}+\|F\|^2_{\mathcal{H}}\right).
\end{array}
\end{equation}
Since $0<\tau<s<\infty$, integrating by parts with respect to $s$ and using the definition of $\mathcal{F}$ given in \eqref{solution1},  we obtain
\begin{equation}\label{t17}
\int_0^{\infty}\mu(s)\mathcal{F}(s)ds=\frac{1}{i\la}\int_0^{\infty}\mu(\tau)f^4(\tau)d\tau +\frac{1}{i\la} \int_0^{\infty}\int_{\tau}^{\infty}\mu^{\prime}(s)e^{-i\la (s-\tau)}f^4(\tau)dsd\tau.
\end{equation}
Now, for the last term on the right hand side in \eqref{t14}, we have
\begin{equation}\label{t17.1}
\left|\Re\left(2i\la c\int_1^2(x-2)y\int_0^{\infty}\mu(s)\overline{\mathcal{F}}_x(s)dsdx\right)\right|\leq 2c\|y\|  \, \left\Vert\la\int_0^{\infty}\mu(s)\overline{\mathcal{F}}_x(s)ds \right\Vert.
\end{equation}
In order to estimate the last term in the above inequality, we will differentiate \eqref{t17} and, using the fact that $f^4\in W$, we get
\begin{equation*}\label{t18}
\begin{array}{ll}
\displaystyle
\left\Vert\la\int_0^{\infty}\mu(s)\overline{\mathcal{F}}_x(s)ds \right\Vert\leq \left\Vert\int_0^{\infty}\mu(\tau)\overline{f^4_x}(\tau)d\tau\right\Vert +{\left\Vert \int_0^{\infty}\int_{\tau}^{\infty}\mu^{\prime}(s)e^{i\la (s-\tau)}\overline{f^4_x}(\tau)dsd\tau \right\Vert}\\
\displaystyle
\leq \left(g(0)\int_1^2 \int_0^{\infty}\mu(s)|f^4_x(\tau)|^2d\tau dx\right)^{\frac{1}{2}}+\left(\int_1^2 \left(\int_0^{\infty}|f^4_x(\tau)|\int_{\tau}^{\infty}|\mu^{\prime}(s)|dsd\tau\right)^2dx\right)^{\frac{1}{2}}
\\
\displaystyle
\leq \sqrt{g(0)}\|f^4\|_{W}+\left(\int_1^2 \left(\int_0^{\infty}\mu(s)|f^4_x(\tau)|d\tau\right)^2dx\right)^{\frac{1}{2}}
\leq 2\sqrt{g(0)c^{-1}}\|F\|_{\mathcal{H}}.
\end{array}
\end{equation*}
Thus, substituting the above inequality into \eqref{t17.1}, we obtain
 \begin{equation}\label{t19}
\left|\Re\left(2i\la c\int_1^2(x-2)y\int_0^{\infty}\mu(s)\overline{\mathcal{F}}_x(s)dsdx\right)\right| \leq 4\sqrt{g(0)c}\left(\|U\|_{\mathcal{H}}\|F\|_{\mathcal{H}}+\|F\|^2_{\mathcal{H}}\right).
\end{equation}
Finally, using \eqref{t15}, \eqref{t16} and \eqref{t19} in \eqref{t14} and thanks to Lemmas \ref{lem1-exp} and \ref{lem3-exp}, we deduce that 
\begin{equation}
cg(0)|y(1)|^2+|\zeta_x(1)|^2\leq e_{11}\left(\|U\|_{\mathcal{H}}\|F\|_{\mathcal{H}}+\|F\|^2_{\mathcal{H}}\right), 
\end{equation}
where $e_{11}=3e_2^{\frac{1}{2}} +4c\mu(0)\sqrt{e_3e_4}+4\sqrt{g(0)c}+cg(0)e_4+e_2$. Thus,  \eqref{est5-exp} and \eqref{est6-exp} hold with $e_9=e_{11}$ and $e_{10}=\frac{e_{11}}{cg(0)}$.
\end{proof}
\begin{Lemma}\label{lem5-exp}
Assume condition \eqref{H}, Hypothesis \ref{Condition1},  $m=1$ and $\ibt{|\la|\geq 1}$. Then,  the solution $(u, v, y, \gamma)^{\top}\in D(\mathcal{A})$ of \eqref{t0} satisfies the following estimates
\begin{equation}\label{est7-exp}
\int_0^1 \eta|u_x|^2dx\leq  e_{12}\left(\|U\|_{\mathcal{H}}\|F\|_{\mathcal{H}}+\|F\|^2_{\mathcal{H}}\right),
\end{equation}
and
\begin{equation}\label{est8-exp}
\int_0^1 \frac{1}{\sigma}|v|^2dx\leq  e_{13}\left(\|U\|_{\mathcal{H}}\|F\|_{\mathcal{H}}+\|F\|^2_{\mathcal{H}}\right),
\end{equation}
where $e_{12}$ and $e_{13}$ are constants independent of $\la$ to be determined below.
\end{Lemma}
\begin{proof}
Using similar arguments as in the Step 1 of the proof of Lemma \ref{lem4_pol}, we obtain
\begin{equation}\label{result1-exp}
\begin{array}{l}
\displaystyle\left(1+\frac{K_a}{2}-M_1\right)\int_0^1 \frac{1}{\sigma}|\la u|^2dx+\left(1-\frac{K_a}{2}- M_2\right)\int_0^1 \eta |u_x|^2dx
\leq e_{14} \left(\|U\|_{\mathcal{H}}\|F\|_{\mathcal{H}}+\|F\|_{\mathcal{H}}^2\right)\\
\displaystyle+\frac{1}{\sigma(1)}|\la u(1)|^2
+\eta(1)|u_x(1)|^2
+\frac{2}{\sigma(1)}|f^1(1)| |\la u(1)|+\frac{K_a}{2}\eta(1)|u_x(1)| |u(1)|,
\end{array}
\end{equation}
where $\displaystyle e_{14}=\frac{2(c_1+1)}{\sqrt{a(1)}}+4c_0c_1+\frac{K_a}{2}c_0c_1+\frac{K_a}{2}c_1$.\\
Thanks to the transmission conditions and Lemma \ref{lem4-exp}, we get
\begin{equation}\label{t20}
\eta(1)|u_x(1)|^2\leq \frac{e_{9}}{\eta(1)} \left(\|U\|_{\mathcal{H}}\|F\|_{\mathcal{H}}+\|F\|^2_{\mathcal{H}}\right)\quad\text{and}\quad|v(1)|^2\leq e_{10} \left(\|U\|_{\mathcal{H}}\|F\|_{\mathcal{H}}+\|F\|^2_{\mathcal{H}}\right).
\end{equation}
By  \eqref{t1} and by  the second estimation in \eqref{t20}, we obtain 
\begin{equation}\label{t22}
|\la u(1)|^2\leq 2|v(1)|^2+2|f^1(1)|^2\leq e_{15}  \left(\|U\|_{\mathcal{H}}\|F\|_{\mathcal{H}}+\|F\|^2_{\mathcal{H}}\right),
\end{equation}
where $\displaystyle e_{15}=2(e_{10}+\max_{x\in[0,1]}\eta^{-1})$.
Now, using Young's inequality we get
\begin{equation}\label{t23}
\frac{2}{\sigma(1)}|f^1(1)| |\la u(1)|\leq \frac{1}{\sigma(1)}|f^1(1)|^2+\frac{1}{\sigma(1)}|\la u(1)|^2\leq e_{16}  \left(\|U\|_{\mathcal{H}}\|F\|_{\mathcal{H}}+\|F\|^2_{\mathcal{H}}\right)
\end{equation}
and
\begin{equation}\label{t23*}
\frac{K_a}{2}\eta(1)|u_x(1)| |u(1)| \leq \tilde{e}_{16} \left(\|U\|_{\mathcal{H}}\|F\|_{\mathcal{H}}+\|F\|^2_{\mathcal{H}}\right),
\end{equation}
where $\displaystyle e_{16}=\frac{1}{\sigma(1)}(e_{15}+\max_{x\in[0,1]}\eta^{-1})$ and $\tilde{e}_{16}=\frac{K_a^2}{8}e_9+\frac{e_{15}}{2}$.
Finally, substituting the first  inequality in \eqref{t20}, \eqref{t22},  \eqref{t23} and \eqref{t23*} into \eqref{result1-exp}, we obtain
\begin{equation}\label{result2-exp}
\begin{array}{l}
\displaystyle\left(1+\frac{K_a}{2}-M_1\right)\int_0^1 \frac{1}{\sigma}|\la u|^2dx+\left(1-\frac{K_a}{2}-M_2\right)\int_0^1 \eta |u_x|^2dx
\leq e_{17} \left(\|U\|_{\mathcal{H}}\|F\|_{\mathcal{H}}+\|F\|_{\mathcal{H}}^2\right),
\end{array}
\end{equation}
with $\displaystyle e_{17}=e_{14}+\frac{e_{9}}{\eta(1)}+\frac{e_{15}}{\sigma(1)}+e_{16}+\tilde{e}_{16}$.
Therefore, using Hypothesis \ref{Condition1} in \eqref{result2-exp}, we obtain \eqref{est7-exp} with $e_{12}=\frac{e_{17}}{1-\frac{K}{2}-M_2}$.\\
From  \eqref{pol4} and \eqref{result2-exp}, we get
\begin{equation*}
\int_0^1 \frac{1}{\sigma} |v|^2dx\leq 2\frac{e_{17}}{1+\frac{K}{2}-M_1}  \left(\|U\|_{\mathcal{H}}\|F\|_{\mathcal{H}}+\|F\|_{\mathcal{H}}^2\right)+2c_0^2\|F\|^2_{\mathcal{H}}.
\end{equation*}
Hence, \eqref{est8-exp} holds with $\displaystyle e_{13}=\frac{2 e_{17}}{1+\frac{K}{2}-M_1}+2c_0^2$.
\end{proof}

\noindent\textbf{Proof of Proposition \ref{proposition4}}.
Adding estimates \eqref{est1_exp}, \eqref{est4_exp}, \eqref{est7-exp} and \eqref{est8-exp} and using the fact that $|\la|\geq 1$, we obtain
\begin{equation}
\|U\|^2_{\mathcal{H}}\leq e_{18}  \left(\|U\|_{\mathcal{H}}\|F\|_{\mathcal{H}}+\|F\|_{\mathcal{H}}^2\right),
\end{equation}
where $\displaystyle e_{18}=me_{1}+e_{4}+e_{12}+e_{13}$.
Thanks to Young's inequality, we get $\|U\|^2_{\mathcal{H}}\leq \left(e_{18}^2+2e_{18}\right) \|F\|_{\mathcal{H}}^2$ and using the fact that $\left(e_{18}^2+2e_{18}\right)\leq \left(e_{18}+1\right)^2$, we deduce that $\|U\|_{\mathcal{H}}\leq K_2 \|F\|_{\mathcal{H}}$  with $K_2=e_{18}+1$ and the thesis follows.\\
\textbf{Proof of Theorem \ref{exponential_Stab}}.  From Proposition \ref{proposition4}, we get \eqref{E1}.  Next, we will prove \eqref{E2}  by a contradiction argument. Suppose that  there exists
$( \la_n, U_n=(u^n, v,^n y^n, \gamma^n)^{\top})\subset \R^{\ast}\times D(\mathcal{A})$, with $|\la_n|\geq 1$ such that  $|\la _n|\to\infty$, $\|U_n\|_{\mathcal{H}}=1$ and there exists a sequence $F_n=(f^1_n,f_n^2, f_n^3, f_n^4)\in\mathcal{H}$ such that
$$\left(i\la_nI-\mathcal{A}\right)U_n=F_n\to 0\quad \text{in}\quad \mathcal{H}\quad\text{as}\quad n\to 0.$$
By Proposition \ref{proposition4} and taking $U=U_n$, $F=F_n$ and $\la=\la_n$,  we can deduce that $\|U_n\|_{\mathcal{H}}\to 0$ as $|\la _n|\to\infty$, which contradicts $\|U_n\|_{\mathcal{H}}=1$. Thus, condition \eqref{E2} holds true. The result follows from the Huang-Pr\"uss Theorem (see  \cite{Huang01} and  \cite{pruss01}) and the proof is thus completed.

\section{\textbf{Appendix}}

\begin{Proposition}(Hardy-Poincar\'e Inequality)\label{Hardy}(see \cite{akil2023stability})
Assume Hypothesis \ref{hyp3}.  Then there exists $C_{HP} > 0$ such that
\begin{equation}\label{HP}\tag{$\rm{HP}$}
\int_0^1 u^2 \dfrac{1}{\sigma}dx\leq C_{HP} \int_0^1 u_x^2dx\quad\forall\, v\in H^1_{\frac{1}{\sigma},0}(0,1).
\end{equation}
where $\displaystyle{C_{HP}=\left(\frac{4}{a(1)}+\max_{x\in [\beta,1]}\left(\frac{1}{a}\right)C_P\right)\max_{x\in [0,1]}\eta(x)}$,  $C_P$ is the constant of the classical Poincar\'e inequality on $(0,1)$ and $\beta\in (0,1)$.  \end{Proposition}

\begin{Lemma}\label{Lemma0}(See Lemma 2.4 in \cite{fragnelli2022linear})
\begin{enumerate}
\item Assume  Hypothesis \ref{hyp1}. If $u\in H^2_{\frac{1}{\sigma}}(0,1)$ and if $v\in H^1_{\frac{1}{\sigma},0}(0,1)$, then $\displaystyle\lim_{x\to 0}v(x)u_x(x)=0$.
\vspace{0.2cm}

\item Assume  Hypothesis \ref{hyp3}. If $u\in D(\mathcal{A})$, then $xu_x(\eta u_x)_x\in L^1(0,1)$.
\vspace{0.4cm}

\item Assume Hypothesis \ref{hyp3}. If $u\in D(\mathcal{A})$ and $K_a\leq 1$, then $\displaystyle \lim_{x\to 0}x |u_x|^2=0$.\vspace{0.2cm}

\item Assume Hypothesis \ref{hyp3}. If $u\in D(\mathcal{A})$,  $K_a> 1$ and $\dfrac{xb}{a}\in L^{\infty}(0,1)$, then  $\displaystyle \lim_{x\to 0}x |u_x|^2=0$.
\vspace{0.2cm}

\item Assume Hypothesis \ref{hyp3}. If $u\in H^1_{\frac{1}{\sigma}}(0,1)$, then $\displaystyle\lim_{x\to 0}\dfrac{x}{a}|u(x)|^2=0$.
\end{enumerate}
\end{Lemma}

\begin{Lemma}\label{lemmaAppendix}
Under Hypothesis \ref{Condition1},  the solution $(u, v, y, \gamma)^{\top}\in D(\mathcal{A})$ of \eqref{pol3} satisfies the following equation
\begin{equation}\label{p20new}
\begin{array}{l}
\displaystyle\left(1+\frac{K_a}{2}\right)\int_0^1 \frac{1}{\sigma}|\la u|^2dx+\left(1-\frac{K_a}{2}\right)\int_0^1 \eta |u_x|^2dx
=\int_0^1\frac{x}{\sigma}\left(\frac{a^\prime-b}{a}\right)|\la u|^2dx+\int_0^1x\frac{b}{a}\eta|u_x|^2dx
\\ 
\displaystyle
+2\Re\left(\int_0^1 f^2\dfrac{x}{\sigma}\overline{u}_xdx\right)
-2\Re\left(i\int_0^1 \left(\frac{xf^1}{\sigma}\right)_x\la \overline{u}dx\right)-\frac{K_a}{2}\Re\left(i\int_0^1\frac{1}{\sigma}f^1\la \overline{u}dx\right)-\frac{K_a}{2}\Re\left(\int_0^1\frac{1}{\sigma}f^2 \overline{u}dx\right)
\\ 
\displaystyle
-\Re\left(\frac{K_a}{2}\eta(1)u_x(1)\overline{u}(1)\right)+\frac{1}{\sigma(1)}|\la u(1)|^2+\eta(1)|u_x(1)|^2+\Re\left(2i\frac{f^1(1)}{\sigma(1)}\la \overline{u}(1)\right).
\end{array}
\end{equation}
\end{Lemma}

\begin{proof}
Multiply \eqref{pol9} by $\displaystyle\frac{K_a}{2\sigma}\overline{u}$,  integrate over $(0,1)$, and take the real part,  we obtain 
\begin{equation}\label{p20**}
\begin{array}{l}\displaystyle
\frac{K_a}{2}\int_0^1\frac{1}{\sigma}|\la u|^2dx-\frac{K_a}{2}\int_0^1\eta| u_x|^2dx=-\Re\left(\frac{K_a}{2}\eta(1)u_x(1)\overline{u}(1)\right)-\frac{K_a}{2}\Re\left(i\int_0^1\frac{1}{\sigma}f^1\la \overline{u}dx\right)\\
\displaystyle
-\frac{K_a}{2}\Re\left(\int_0^1\frac{1}{\sigma}f^2 \overline{u}dx\right).
\end{array}
\end{equation}
Multiplying equation \eqref{pol9} by $\displaystyle\frac{-2x}{\sigma}\overline{u}_x$,  integrating over $(0,1)$,  and taking the real part,  we have
\begin{equation}\label{p21}
\begin{array}{l}
\displaystyle\int_0^1 \left(\frac{x}{\sigma}\right)^{\prime}|\la u|^2dx-\frac{1}{\sigma(1)}|\la u(1)|^2+\lim_{x\to 0}\dfrac{x}{\sigma(x)}|\la u(x)|^2-2\Re\left(\int_0^1 (\eta u_x)_x x \overline{u}_xdx\right)
\\ 
\displaystyle =2\Re\left(\int_0^1 f^2\dfrac{x}{\sigma}\overline{u_x}dx\right)-2\Re\left(i\int_0^1 \left(\frac{x f^1 }{\sigma}\right)_x\la \bar{u}dx\right)+\Re\left(2i\frac{f^1(1)}{\sigma(1)}\la \overline{u}(1)\right)
\\ 
\displaystyle
-2\lim_{x\to 0}\Re\left[i\la f^1(x)\frac{x}{\sigma(x)}\bar{u}(x)\right].
\end{array}
\end{equation}
Clearly, $\displaystyle\left(\frac{x}{\sigma}\right)^{\prime}=\frac{1}{\sigma}-\frac{x}{\sigma}\left(\frac{a^\prime-b}{a}\right)$ and $\eta'=\dfrac{b}{a}\eta$,  
then 
\begin{equation}\label{p22}
 \int_0^1 \left(\frac{x}{\sigma}\right)^{\prime}|\la u|^2dx=\int_0^1 \frac{1}{\sigma}|\la u|^2dx-\int_0^1\frac{x}{\sigma}\left(\frac{a^\prime-b}{a}\right)|\la u|^2dx
\end{equation}
and 
\begin{equation}\label{p23}
\begin{array}{l}
\displaystyle
-2\Re\left(\int_0^1 (\eta u_x)_x x \bar{u}_xdx\right)=2\Re\left(\int_0^1\eta u_x(x \bar{u}_x)_xdx\right)-2\Re\left[x\eta |u_x|^2\right]_0^1=2\int_0^1\eta |u_x|^2dx \\ \displaystyle -\int_0^1 (x\eta )^{\prime}|u_x|^2dx

-\Re\left[x\eta |u_x|^2\right]_0^1=\int_0^1 \eta|u_x|^2dx-\int_0^1 x \frac{b}{a}\eta|u_x|^2dx
\\ \displaystyle
-\eta(1)|u_x(1)|^2+\lim_{x\to 0}x\eta(x) |u_x(x)|^2.
\end{array}
\end{equation}
Using Lemma \ref{Lemma0},  we deduce
\begin{equation}\label{p24}
\lim_{x\to 0}\dfrac{x}{\sigma(x)}|\la u(x)|^2=0, \quad\lim_{x\to 0}x \eta(x) |u_x(x)|^2=0, \quad \text{and}\quad\lim_{x\to 0}\Re\left[i\la f^1(x)\frac{x}{\sigma(x)}\bar{u}(x)\right]=0.
\end{equation}
Thus, using equations \eqref{p22}-\eqref{p24} in \eqref{p21},  we obtain 
\begin{equation}\label{p20}
\begin{array}{l}
\displaystyle\int_0^1 \frac{1}{\sigma}|\la u|^2dx+\int_0^1 \eta |u_x|^2dx
=\int_0^1\frac{x}{\sigma}\left(\frac{a^\prime-b}{a}\right)|\la u|^2dx+\int_0^1x\frac{b}{a}\eta|u_x|^2dx+2\Re\left(\int_0^1 f^2\dfrac{x}{\sigma}\overline{u}_xdx\right)
\\ 
\displaystyle
-2\Re\left(i\int_0^1 \left(\frac{xf^1}{\sigma}\right)_x\la \overline{u}dx\right)+\frac{1}{\sigma(1)}|\la u(1)|^2+\eta(1)|u_x(1)|^2+\Re\left(2i\frac{f^1(1)}{\sigma(1)}\la \overline{u}(1)\right).
\end{array}
\end{equation}
Finally, summing \eqref{p20} and \eqref{p20**}, we obtain \eqref{p20new}.

\end{proof}

\begin{Proposition}\label{prop_1}(See \cite{Zhang-ZAMP}, Proposition 2.2 (ii))
Let $\displaystyle \mathcal{F}(s)=\int_0^s e^{-\xi(s-\tau)}f^4(\tau)d\tau$, where $f^4\in W$ and $\xi\in \mathbb{C}$. Then, for any $\Re(\xi)=0$, there exists $m_1>0$ such that
\begin{equation}
\|\mathcal{F}(s)\|_{W}\leq m_1 \|f^4(s)\|_{W}.
\end{equation}
\end{Proposition}

\begin{Proposition}\label{prop_2}(See \cite{Zhang-ZAMP})
Assume condiftion \eqref{H}. Then, for any $\epsilon>0$ there exists a constant $m_2>0$ such that
$$ \inf_{\la\in\R; |\la|\geq \epsilon>0}\int_0^{\infty}\mu(s)|1-e^{-i\la s}|ds\geq m_2.
$$
\end{Proposition}

\section*{\textbf{Conclusion}}
This work aims to examine the stabilization of the transmission problem of degenerate wave equation and heat equation, specifically in relation to the Coleman-Gurtin heat conduction law or Gurtin-Pipkin law with memory effect. In this study, we examine the polynomial stability of the system utilizing the Coleman-Gurtin heat conduction model. We establish that the system shows a decay rate of the kind $t^{-4}$. Afterwards, we prove that, when Gurtin-Pipkin heat conduction is employed, the system remains exponentially stable. Regarding the optimality of the decay rate we conjecture it is optimal since it is complicated to prove it with general consideration of our functions $a$ and $b$. 

\section*{\textbf{Acknowledgment}}
The authors would like to thank the Project Horizon Europe Seeds {\it STEPS: STEerability and 
controllability of PDES in Agricultural and Physical models} and the Project {\it PEPS JCJC-FMHF- FR2037}.\\
Genni Fragnelli is also a member of the  {\it Gruppo Nazionale per l'Analisi Ma\-te\-matica, la Probabilit\'a e le loro Applicazioni (GNAMPA)} of the Istituto Nazionale di Alta Matematica (INdAM) and a member of {\it UMI ``Modellistica Socio-Epidemiologica (MSE)''}. She is partially supported by the GNAMPA project 2023 {\em Modelli differenziali per l'evoluzione del clima e i suoi impatti} and by FFABR {\it Fondo per il finanziamento delle attivit\`a base di ricerca} 2017. 

\section*{\textbf{Declarations}}
\textbf{Competing interest} The authors have not disclosed any competing interests.


\end{document}